\documentclass[a4paper,11pt]{amsart}

\usepackage{amsfonts}
\usepackage{amssymb}
\usepackage{graphicx}%

\usepackage{amscd}
\usepackage{mathrsfs}
\newtheorem{theorem}{Theorem}[section]

\newtheorem{corollary}[theorem]{Corollary}
\newtheorem{definition}[theorem]{Definition}

\newtheorem{remark}[theorem]{Remark}

\numberwithin{equation}{section}

\def\DO{\mathcal D}
\def\RE{\mathbb R}
\def\CO{{\mathbb C}}
\def\N{\mathbb N}

\def\C{\mathcal C}

\def\K{\mathcal K}
\def\A{\mathcal A}

\def\W{\mathcal W}
\def\F{\mathcal F}
\def\wh{\widehat}

\begin{document}

\title[Ordinary Differential Equations with singular coefficients]
{Ordinary differential equations with singular coefficients: an
intrinsic formulation with applications to the Euler-Bernoulli
beam equation}

\author{Nuno Costa Dias}

\author{Cristina Jorge}

\author{Jo\~{a}o Nuno Prata}

\begin{abstract}
We study a class of linear ordinary differential equations (ODE)s with distributional coefficients. These
equations are defined using an {\it intrinsic} multiplicative product of Schwartz distributions which is an
extension of the H\"ormander product of distributions with non-intersecting singular supports [L. H\"ormander,
The Analysis of Linear Partial Diffe\-rential Operators I, Springer-Verlag, 1983]. We provide a regularization procedure for these ODEs and prove an existence and
uniqueness theorem for their solutions. We also determine the conditions for which the solutions are
regular and distributional. These results are used to study the Euler-Bernoulli beam
equation with discontinuous and singular coefficients. This problem was addressed in the past using intrinsic
products (under some restrictive conditions) and the Colombeau formalism (in the general case). Here we present
a new intrinsic formulation that is simpler and more general. As an application, the case of a non-uniform static
beam displaying structural cracks is discussed in some detail.
\end{abstract}

\maketitle



\section{Introduction}

In this paper we study a class of ordinary differential equations (ODE)s {\it formally} of the form
\begin{equation}
\sum_{i=0}^n  c_i \psi^{(i)} =f \, , \label{eq01}
\end{equation}
where $f$ is a smooth function and the coefficients $c_i$ belong to the space of distributions $\A= \cup_{i=0}^\infty D^i_x [\C_p^{\infty}] \subset \DO'$, where $C_p^\infty$ is the space of piecewise smooth functions with support on $\RE$, $D^i_x$ is the $i$th-order distributional derivative, and $\DO'$ is the space of Schwartz distributions. The coefficients $c_i \in \A$ can be written explicitly in the form $c_i=f_i+\Delta_i$, where $f_i \in \C_p^\infty$ and $\Delta_i \in \DO'$ is a distribution of finite support. 

ODEs formally of this form appear naturally in models of non-smooth systems and of systems with singularities (e.g. systems with point interactions in quantum mechanics \cite{Albeverio1,Albeverio2,DPJ16}, beams with structural cracks \cite{Cad08,HO07,YSR01} in the classical theory of solids, etc). The main problem in these cases is how to define the formal equation (\ref{eq01}) precisely. Notice that for $c_i \in \A$, the ODE (\ref{eq01}) does not in general display smooth solutions and, unless some additional structure is introduced, it is not well-defined for non-smooth functions $\psi$ either (because the terms $c_i \psi^{(i)}$ may involve a product of two distributions).

This problem has been studied in several different cases (i.e. for particular sets of discontinuous or singular coefficients) using a variety of different approaches. Most significative are the
generalized functions formulation in the sense of Colombeau \cite{Col84,CLNP89,GKOS01,Obe92}, the formulations in terms of distributions acting on discontinuous test functions \cite{Albeverio2,Kurasov,KurasovBoman}, and the {\it
intrinsic} approaches using suitable multiplicative products of Schwartz distributions, e.g.
\cite{Cad08,DPJ16,HO07,Obe92,Sarrico,YSR01}\footnote{Hilbert space methods are also very important, namely in
the context of singular perturbations of Schr\"odinger operators \cite{Albeverio1,Albeverio2,Dias3,Golovaty}.}. In the latter case the
entire formulation is strictly defined within the standard space of Schwartz distributions. This considerably simplifies
the formulation when compared to the approaches based on generalized functions or more general distributions. However, it also restricts the type of
problems that can be considered, and the type of solutions that are admissible. For instance, model products,
which are the most general products in the hierarchy given by M. Oberguggenberger in (section 7, \cite{Obe92}),
have been used to formulate ODEs with discontinuous coefficients. While the formalism is compatible with
non-smooth solutions, it is not in general well-defined for discontinuous ones. This is clearly discussed in
\cite{HO07} where the formulation of the Euler-Bernoulli beam (EBB) equation with discontinuous coefficients was
studied in detail.

In this paper, we will study the following {\it intrinsic} formulation of (\ref{eq01}):
\begin{equation}
\sum_{i=0}^n \left(a_i(x) * \psi^{(i)}(x) + \psi^{(i)}(x) * b_i(x)\right)=f(x)  \, ,\label{eq0}
\end{equation}
for the case where $f$ is a smooth function, and the coefficients $a_i,b_i \in \A$. The key structure in (\ref{eq0}) is the intrinsic product of
distributions $*$ that was defined in \cite{DP09}. This product is an extension of the H\"ormander product
of distributions with non-intersecting singular supports (pag.55, \cite{Hor83}). It is associative, extends the
standard product of smooth functions, satisfies the Leibnitz rule and it is an inner operation in $\A$.
Moreover, it is non-commutative, which is the reason why the ODE (\ref{eq0}) has the left and right coefficients
$a_i$ and $b_i$, respectively.

Equipped with the product $*$, the
space $\A$ becomes a differential algebra of distributions that satisfies {\it all} the properties stated in the Schwartz impossibility theorem
\cite{Sch54,Ros87}. In fact, it is (essentially) the unique differential algebra that satisfies all these properties and contains the space $\C_p^\infty$ \cite{DJP16-2}.
There is no contradiction with Schwartz's result because $\A$ is only a subspace of $ \DO'$; which is however sufficiently large to
allow for a precise formulation of an interesting class of differential problems with distributional coefficients.
Notice that the differential expression (\ref{eq0}) with coefficients $a_i,b_i \in
\A$ is well-defined for all $\psi \in \A$, and thus the ODEs (\ref{eq0}) admit distributional coefficients,
including the Dirac measure and all its derivatives and, in general, may display discontinuous and
distributional solutions. Moreover, we will see that if $a_i, b_i$ are smooth then (\ref{eq0}) reduces to (\ref{eq01}) with
$c_i=a_i+b_i$, and thus the new equations generalize the standard linear ODEs.

The equations of the form (\ref{eq0}) have already been studied in the recent papers \cite{DP09,DPJ19}, but only for the case of piecewise smooth solutions. The aim of the present paper is to go one step further and to study the properties of (\ref{eq0}) for the most general case of arbitrary coefficients $a_i,b_i \in \A$ and distributional solutions $\psi \in \A$. In addition, the new formalism will be used to study the EBB equation with piecewise smooth or singular coefficients. Finally, we will also establish an interesting connection between the ODEs of the form (\ref{eq0}) and a class of functional equations that we shall call {\it limit ODE}s. 

Let us explain this last result in some detail. Let us re-write (\ref{eq0}) in the form $\wh L \psi =f$ where $\wh L : \A \longrightarrow \A$.
We will show that in the general case (i.e. for $a_i,b_i \in \A$) $\wh L$ is the weak operator limit of a large class of one-parameter families of operators
\begin{equation}\label{SmoothL}
\wh L_\epsilon =\sum_{i=0}^n \left( a_{i\epsilon} + b_{i \epsilon} \right) D_x^i \quad , \quad \epsilon >0
\end{equation}
with smooth coefficients $a_{i\epsilon}, b_{i \epsilon}$ that satisfy, in the sense of distributions,
\begin{equation} \label{SmoothCoef}
\lim_{\epsilon \downarrow 0} a_{i\epsilon} = a_i \qquad , \qquad  \lim_{\epsilon \downarrow 0} b_{i \epsilon}= b_i \, .
\end{equation}
It follows that $\psi$ is a (generalized) solution of (\ref{eq0}) iff it is a solution of the {\it limit ODE}:
\begin{equation} \label{LimitODE}
\lim_{\epsilon \downarrow 0 } \left( \wh L_\epsilon \psi \right) =f
\end{equation}
 for one (and thus for all) of the one-parameter families of operators $\wh L_\epsilon $ in the previous class. We thus conclude that the equation (\ref{eq0}) provides an approximation for the entire class of differential equations with smooth (possibly sharply concentrated) coefficients $\wh L_\epsilon \psi =f$, and an equivalent formulation of the limit differential equation (\ref{LimitODE}) which is manifestly independent of the particular sequence $\wh L_\epsilon \overset{w}{\longrightarrow} \wh L$.   
 

 
Here is a brief summary of our results: In the first part of the paper we study the main properties of the equations (\ref{eq0}) and (\ref{LimitODE}). The two equations are proved to be equivalent (for suitable sequences of smooth coefficients (\ref{SmoothCoef})) in Theorem \ref{Theorem_WL} and Corollary \ref{EqLimitODE}. The conditions for which their solutions are regular or singular are determined in Theorem \ref{1i}, and the interface conditions satisfied by the regular solutions are studied in Theorems \ref{1ii} and \ref{1iii}. Finally, an existence and uniqueness result for their solutions is proved in Theorem \ref{1iv} and Corollary \ref{Corollary-1iv}. A simple example is solved explicitly in section 4, in order to illustrate these results.  

In the second part of the paper, the new formalism is used to study the EBB equation with discontinuous and/or singular coefficients. The two cases are natural to consider in models of beams made of different sections, and of beams
with structural cracks \cite{BC07,Cad08,YSM00,YSR01}. Model products \cite{Obe92} have been used in this context, and shown not to be compatible with the case of singular coefficients \cite{HO07}. Instead, this case has
been formulated using other particular products under restrictive conditions \cite{Bag95,Bag02,BC07,Cad08} or,
alternatively, the formalism of generalized functions \cite{HO09,HKO13}. In this paper we provide a new and more general
intrinsic formulation, allowing for a unified treatment of the physically most relevant cases. As an application, several different types of beams are studied, including the case of beams with a structural crack at the contact point of two different segments. Up to our knowledge, this case has never been considered in the literature.



\subsection*{Notation} $\Omega$ and $\overline\Omega$ denote an arbitrary open interval of $\RE$ and its closure,
respectively. The functional spaces are denoted by calligraphic capital letters ($\A (\Omega)$, $\C(\Omega)$,
$\DO'(\Omega)$,...). If $\Omega =\RE$ we write only $\A$, $\C$, $\DO'$,... unless we want to emphasize that the
support is $\RE$.

$H$ is the Heaviside step function and $H_-=1-H$. Moreover, $\delta(x-x_0)$ is the Dirac measure with support at
$x_0$. If $x_0=0$ we sometimes write only $\delta$.

In general, we do not distinguish a locally integrable function from the associated regular
distribution (the only exception is in Definition 3.1, where we write $\phi_{\DO'}$ to denote the regular
distribution associated to the smooth function $\phi$).

The $n$th-order (Schwartz) distributional derivative of $\psi$ is written $D_x^n \psi$ or $\psi^{(n)}$. Letters
with a hat are operators.

\begin{section}{A multiplicative product of Schwartz distributions}

In this section we review some basic notions about Schwartz distributions and present the main properties of the
multiplicative product $*$. For details and proofs the reader should refer to \cite{DP09,DPJ16}. We also discuss a smooth regularization of the product $*$, and prove a new result (Theorem \ref{Theorem_WL}) that will be used in the next section to prove the equivalence of eqs.(\ref{eq0}) and (\ref{LimitODE}).  

\subsection{The algebra of distributions $\A$}

We start with some basic notation. Let $\DO(\Omega)$ denote the space of smooth functions with support on a
compact subset of $\Omega$ and let $\DO'(\Omega)$ be its dual, the space of Schwartz distributions. If $\Omega
=\RE$, we write simply $\DO'$. Let $F|_{\Omega}$ denote the restriction of $F \in \DO'$ to the space
$\DO(\Omega)$. We have, of course, $F|_{\Omega} \in \DO'(\Omega)$. The singular support of a distribution $F \in
\DO'$ (denoted sing supp $F$) is, as usual, the closed set of points where $F$ is not a smooth function.

An useful concept is the order of a distribution \cite{Kan98}: we say that $F \in \DO'$ is of {\it order} $n$
(and write $n=$ ord $F$) iff $F$ is the $n$th order distributional derivative (but not a lower order
distributional derivative) of a regular distribution.

Finally, let $\C_p^{\infty}$ be the space of piecewise smooth functions on $\RE$: $\psi \in \C_p^{\infty}$ iff
there is a finite set $I\subset \RE$ such that $\psi \in \C^{\infty}(\RE \backslash I)$ and the lateral limits
$\lim_{x \to x_0^{\pm}} \psi^{(j)} (x) $ exist and are finite for all $x_0 \in I$ and all $j \in \N_0$.

A distributional extension of the space $\C_p^{\infty}$ is given by:

\begin{definition}
Let $\A$ be the space of all functions in $\C_p^{\infty}$ - regarded as Schwartz distributions - together with
all their distributional derivatives to all orders. Moreover, for $\Omega \subset \RE$ an open set, the space of
distributions of the form $F|_{\Omega}$, where $F \in \A$, is denoted by $\A(\Omega)$.
\end{definition}

We have $\C_p^{\infty} \subset \A \subset \DO'$. All the elements of $\A$ are distributions with finite singular
support. They can be written in the form $F=\Delta_F +f$, where $\Delta_F$ is a distribution with finite support
(i.e. a finite linear combination of Dirac deltas and their derivatives) and $f \in \C_p^{\infty}$. The next
Theorem states this property more precisely:

\begin{theorem} \label{Theorem_A}
$F \in \A$ iff there is a finite set $I=\{x_1,...,x_m\} \subset \RE$ (where $x_i<x_k$ for $i<k$) associated with
a set of open intervals $\Omega_i=(x_i,x_{i+1})$, $i=0,..,m$ (where $x_0=-\infty$ and $x_{m+1}=+\infty$) such
that ($\chi_{\Omega_i}$ is the characteristic function of $\Omega_i$):
\begin{equation}\label{FormF}
F= \sum_{i=1}^m \sum_{j=0}^n c_{ij}\delta^{(j)}(x-x_i) + \sum_{i=0}^m f_i \chi_{\Omega_i}
\end{equation}
for some $c_{ij} \in \CO$ and $f_i \in \C^{\infty}(\RE)$. We have, of course, sing supp $F \subseteq I$.
\end{theorem}

The product $*$ will be defined in the space $\A$. Let us first recall some basic definitions about products of
distributions. Let $\Xi \subseteq \RE$ be an open set. The dual product of $F \in \DO'(\Xi)$ by $g \in
\C^\infty(\Xi)$ is defined by
\begin{equation} \label{prod1}
\langle F\cdot g, t \rangle= \langle F, gt \rangle \quad , \quad \forall t \in \DO(\Xi)
\end{equation}

The H\"ormander product of distributions extends the dual product to the case of two distributions with
finite and disjoint singular supports (pag.55, \cite{Hor83}).

\begin{definition}
Let $F,G \in \A$  be two distributions with finite disjoint singular supports. Then there exists a finite open
cover of $\RE$ (denote it by $\{\Xi_i \subset \RE,\, i=1,..,d \}$) such that, on each open set $\Xi_i$, either
$F$ or $G$ is a $\C^{\infty}(\Xi_i)$-function. Hence, on each $\Xi_i$, the two distributions can be multiplied
using the dual product (\ref{prod1}). The H\"ormander product of $F$ by $G$ is then defined as the unique
distribution $F \cdot G \in \A$ that satisfies:
$$
\left(F \cdot G\right)|_{\Xi_i}= F|_{\Xi_i} \cdot G|_{\Xi_i} \quad , \quad  i=1,..,d.
$$
where the product on the right hand side is the dual product (we will use the same notation for the H\"ormander and the dual product since one is a trivial extension of the other).

\end{definition}

The new product $*$ extends the H\"ormander product to the case of an arbitrary pair of distributions in $\A $:

\begin{definition}
The multiplicative product $*$ is defined for all $F,G \in \A$ by:
\begin{equation} \label{prod}
F * G= \lim_{\epsilon \downarrow 0} F(x) \cdot G(x+\epsilon),
\end{equation}
where the product in $F(x) \cdot G(x+\epsilon)$ is the H\"ormander product and the limit is taken in the
distributional sense.
\end{definition}

The explicit form of $F*G$ is given in Theorem \ref{2.5} below, and the main properties of $*$ are stated in
Theorem \ref{2.7}. Let $F,G \in \A$ and let $I_F$ and $I_G$ be the singular supports of $F$ and $G$,
respectively. Let $I=I_F \cup I_G$ and write explicitly $I=\{x_1,..,x_m\}$ (where $x_i<x_k$, for $i<k$). Define
the open sets $\Omega_i=(x_i,x_{i+1})$, $i=0,..,m$ (with $x_0=-\infty$ and $x_{m+1}=+\infty$). Then, in view of
Theorem \ref{Theorem_A}, $F$ and $G$ can be written in the form:
\begin{eqnarray}\label{1}
F &=& \sum_{i=1}^m \sum_{j=0}^n a_{ij}\delta^{(j)}(x-x_i) + \sum_{i=0}^m f_i \chi_{\Omega_i} \nonumber \\
G &=& \sum_{i=1}^m \sum_{j=0}^n b_{ij}\delta^{(j)}(x-x_i) + \sum_{i=0}^m g_i \chi_{\Omega_i}
\end{eqnarray}
where $f_i,g_i \in \C^\infty$ and $a_{ij}=0$ if $x_i \notin I_F$ or if $j \ge$ ord $F$, and likewise for
$G$. Then we have:
\begin{theorem}\label{2.5}
Let $F,G \in \A$ be written in the form (\ref{1}). Then $F*G$ is given explicitly by
\begin{equation} \label{prodf}
F * G = \sum_{i=1}^m \sum_{j=0}^n \left[ a_{ij} g_i (x) + b_{ij} f_{i-1}(x) \right] \cdot \delta^{(j)}(x-x_i) +
\sum_{i=0}^m f_i g_i \chi_{\Omega_i}.
\end{equation}
and $F*G \in \A$.
\end{theorem}

A simple Corollary of this Theorem is:

\begin{corollary} \label{2.6T}

Let $\Omega\subset \RE$ be an open set, and $F,G \in \A$ be such that $F|_\Omega = f \in \C^{\infty}(\Omega)$.
Then
\begin{equation} \label{2.6E}
\left(F*G)\right|_\Omega =\left(G*F)\right|_\Omega= f \cdot \left(G|_\Omega \right) \, .
\end{equation}

\end{corollary}

Other simple results that follow from (\ref{prodf}) are:
\begin{eqnarray} \label{prods}
& H(x) * \delta^{(i)}(x)  =  \delta^{(i)}(x)* H_-(x)=0 & \nonumber \\
& H_-(x) * \delta^{(i)}(x) = \delta^{(i)}(x)* H(x)= \delta^{(i)}(x) & \\
& \delta^{(i)}(x-x_0) *\delta^{(j)}(x-x_1)=0 & \nonumber
\end{eqnarray}
Here, $H$ is the Heaviside step function ($H(x)=1$ for $x\ge 0$, and $H(x)=0$ for $x<0$), $H_-(x)=1-H(x)$,
$x_0,x_1 \in \RE$ and $i,j \in \N_0$.

Finally, the main properties of $*$ are summarized in the following

\begin{theorem} \label{2.7}
The product $*$ is an inner operation in $\A$, it is associative, distributive and non-commutative. Moreover, it
reproduces the H\"ormander product of distributions if the singular supports of $F$ and $G$ do not intersect,
and the standard product of functions if $F$ and $G$ are regular distributions. In $\A$, the distributional
derivative $D_x$ is an inner operator and satisfies the Leibnitz rule with respect to the product $*$.
\end{theorem}

Hence, the space $\A$ endowed with the product $*$ becomes an associative, noncommutative differential algebra
of distributions.

\subsection{Smooth regularization of the product $*$}

In view of Theorem \ref{Theorem_A} every $F \in \A$ can be written in the form $F=f + \Delta$ where $f \in C_p^\infty$ and $\Delta=\sum_{i,j} c_{ij} \delta^{(j)}(x-x_i)$. For each $F \in \A$ we can then define the following {\it associated} one-parameter families of smooth functions $F_{\epsilon}^-$ and $F_{\epsilon}^+$, which converge to $F$ in $\DO'$ as $\epsilon \to 0^+$. 

\begin{definition} \label{Smoothaprox}

Let $F=f + \Delta \in \A$. Let $I_f=$ sing supp $f$ (which is a finite set) and define $I_f(\epsilon)= \cup_{x\in I_f}[x-\epsilon,x+\epsilon]$, $\epsilon >0$. 

For some $\epsilon_0>0$, let $\left(f_\epsilon\right)_{0<\epsilon \le \epsilon_0}$ be a one-parameter family of smooth functions such that:
\begin{enumerate}

\item[(C1)] $f_\epsilon(x)=f(x) , \,\, \forall \, x\notin I_f(\epsilon)$.  

\item[(C2)] The functions $f_\epsilon$ are uniformally bounded on the sets $I_f(\epsilon)$, i.e. there exists $M>0$ such that for all $0<\epsilon \le \epsilon_0$:
$$
|f_\epsilon(x)| \le M  \, \, ,  \, \, \forall x\in I_f(\epsilon)  \, .
$$ 
\end{enumerate}

Moreover, for each $x_i \in$ supp $\Delta$, and $0<\epsilon \le \epsilon_0$, let $v_{x_i\epsilon}$ be a smooth, non-negative function such that:
\begin{enumerate}
	\item[(C3)] supp $v_{x_i\epsilon} \subseteq [x_i-\epsilon, x_i + \epsilon]$, 
	\item[(C4)] $\int v_{x_i\epsilon}(x) \, dx =1$, 
\end{enumerate} 
and define $v_{\epsilon}= \sum_{i,j} c_{ij} (v_{x_i\epsilon}) ^{(j)}$, where the coefficients $c_{ij}$ are the ones in $\Delta=\sum_{i,j} c_{ij} \delta^{(j)}(x-x_i)$.

Finally, let $F_\epsilon (x)= f_\epsilon(x) + v_\epsilon(x)$, and define the right and left shifts of $F_\epsilon$: 
\begin{equation}
F_{\epsilon}^+ (x)=F_\epsilon (x-\epsilon) \quad , \quad F_{\epsilon}^- (x)=F_\epsilon (x+\epsilon) 	\, .
\end{equation}

For a given $F \in \A$, the set of {\it associated} one-parameter families of functions of the form $\left(F_{\epsilon}^- \right)_{0<\epsilon \le\epsilon_0}$ is denoted by $\F_-(F)$, while the set of one-parameter families of functions of the form  
$\left(F_{\epsilon}^+ \right)_{0<\epsilon \le\epsilon_0}$ is denoted by $\F_+(F)$.

\end{definition}

Before we proceed let us also define the following operators. Let $F \in \A$. Then
\begin{equation}\label{Aop}
\wh F_+: \A \longrightarrow \A; \, \wh F_+ \psi= F * \psi \qquad , \qquad \wh F_-: \A \longrightarrow \A; \, \wh F_- \psi= \psi * F
\end{equation}
If $F \in \C^{\infty}$ then the previous operators are both identical to:
\begin{equation}\label{Smoothop}
\wh F: \A \longrightarrow \A; \, \wh F \psi= F \cdot \psi
\end{equation}
where $\cdot$ is the dual product. 

We then have:
\begin{theorem} \label{Theorem_WL}

Let $F \in \A$ and let 
$$
\left(F_{\epsilon}^- \right)_{0<\epsilon \le\epsilon_0} \in \F_-(F) \quad \mbox{and} \quad \left(F_{\epsilon}^+ \right)_{0<\epsilon \le\epsilon_0} \in \F_+(F)
$$ 
be two one-parameter families of 
smooth functions associated to $F$. Then, in the sense of distributions: 
\begin{equation}\label{Dlimit}
\lim_{\epsilon \downarrow 0 } \, {F_{\epsilon}^\pm} = F
\end{equation}
Moreover:
\begin{equation}\label{weaklimit}
{\rm w}\lim_{\epsilon \downarrow 0} \, \wh {F^\pm_{\epsilon}} = \wh F_\pm
\end{equation}
where $\rm{wlim}$ denotes the weak operator limit, the operators $\wh{F^\pm_{\epsilon}}=F^\pm_{\epsilon} \cdot $ are (for each $\epsilon$) of the form (\ref{Smoothop}), and the operators $\wh F_\pm$ are given by (\ref{Aop}).

\end{theorem}

\begin{proof}

The two identities (\ref{Dlimit}) and (\ref{weaklimit}) were already proved in [Theorem 3.3, \cite{DPJ16}] for the case $F=\Delta=\sum_i c_i\delta^{(i)}(x)$. The extension to the case where $\Delta$ has support on more than one point (but on a finite set) was also discussed in \cite{DPJ16} and is strainghtforward. 

We then focus on the remaining case $F=f \in \C_p^\infty$, and consider the simplest example where sing supp $f=\{0\}$. The proof of the general case where sing supp $f$ is an arbitrary finite set follows exactly the same steps. We divide the proof in two parts:
 
1) Proof of eq.(\ref{Dlimit}). We have to show that $F_{\epsilon}^\pm(x) =f_\epsilon (x\mp \epsilon) \overset{\DO'}{\longrightarrow} f$ as $\epsilon \to 0^+$. The action of $F_{\epsilon}^+(x)$ on an arbitrary test function $t \in \DO$ yields:
\begin{eqnarray*}
&& \lim_{\epsilon \downarrow 0} \langle f_\epsilon (x - \epsilon),t \rangle =
\lim_{\epsilon \downarrow 0} \int f_\epsilon (x - \epsilon) t(x) dx \\
&=&  \lim_{\epsilon \downarrow 0} \left[ \int f (x) t(x+ \epsilon) dx - \int_{-\epsilon}^\epsilon f(x) t(x+ \epsilon) dx + \int_{-\epsilon}^\epsilon f_\epsilon (x) t(x+ \epsilon) dx \right]
\end{eqnarray*}
where we used the property (C1) from Definition \ref{Smoothaprox}.
The integrands in the second and third integrals are bounded functions and thus the limit $\epsilon \to 0^+$ of these integrals is zero. By dominated convergence of the first integral we then have:
$$
\lim_{\epsilon \downarrow 0} \langle f_\epsilon (x - \epsilon),t \rangle = \int f (x) t(x) dx = 
 \langle f,t \rangle
$$
which proves (\ref{Dlimit}) for $F_{\epsilon}^+(x)$. An equivalent result is valid for 
$F_{\epsilon}^-(x)$. This concludes the proof of (\ref{Dlimit}).

2) Proof of eq.(\ref{weaklimit}). For $F=f$, eq.(\ref{weaklimit}) reads (for the case $\wh F_+$):
\begin{eqnarray} \label{Case-}
& & {\rm w}\lim_{\epsilon \downarrow 0} \, \wh {F_{\epsilon}^+} = \wh F_+ \, \,   \Longleftrightarrow \, \, 
{\rm w}\lim_{\epsilon \downarrow 0} \, f_{\epsilon}(x-\epsilon) \cdot = f * \\
& \Longleftrightarrow & \lim_{\epsilon \downarrow 0} \langle
f_\epsilon(x-\epsilon) \cdot \psi(x) , t\rangle = \langle f * \psi, t \rangle \, , \quad \forall \psi \in \A , \quad \forall t \in \DO \, .\nonumber 
\end{eqnarray}
Let us write $\psi = g +\Xi$ where $g \in \C_p^{\infty}$ and $\Xi$ is of finite support. We then consider the two cases $\psi=g$ and $\psi =\Xi$ separately:

2.1) For $\psi =g$ we have:
$$
\langle
f_\epsilon(x-\epsilon) \cdot g(x) , t \rangle 
= \int f_\epsilon(x-\epsilon) g(x) t(x) \, dx 
$$
Setting $h=gt$ then $h \in \C_p^{\infty}$ and is of compact support. The previous integral yields:
$$
\int f_\epsilon(x) h(x+\epsilon) \, dx
= \int f (x) h(x+ \epsilon) dx - \int_{-\epsilon}^\epsilon f(x) h(x+ \epsilon) dx + \int_{-\epsilon}^\epsilon f_\epsilon (x) h(x+ \epsilon) dx 
$$
where we used the property (C1) from the definition of $F_{\epsilon}^+$. The limit $\epsilon \to 0^+$ of the second and third integrals is zero (because the integrands are bounded) and the limit of the first integral yields (by dominated convergence):
$$
\lim_{\epsilon \downarrow 0} \int f (x) h(x+ \epsilon) dx = \int f (x) h(x) dx =  \langle
f * g , t \rangle 
$$
where we have used the fact that the $*$ product of regular distributions reproduces the standard product of functions (cf. Theorem \ref{2.7}). Hence:
\begin{equation}\label{Caseg-}
\lim_{\epsilon \downarrow 0}\langle
 f_\epsilon(x-\epsilon) \cdot g(x) , t \rangle =\langle
f * g , t \rangle \, .
\end{equation}

2.2) Now consider the remaining case $\psi=\Xi$. Since supp $\Xi$ is a finite set, we can write $\Xi= \sum_i \Xi_i$ where supp $\Xi_i = \{x_i \}$, $x_i \in \RE$. We then have to calculate:
$$
\langle
f_\epsilon(x-\epsilon) \cdot \Xi , t \rangle =\sum_i \langle
f_\epsilon(x-\epsilon) \cdot \Xi_i , t \rangle \, .
$$
Since $f \in \C_p^{\infty} \cap \C^{\infty} (\RE \backslash \{0\})$, we can write it in the form $f=H_-f_- +Hf_+$ where $f_-,f_+ \in \C^\infty(\RE)$. Let us assume that $x_i \le 0$. For $x\le 0$ we have $x-\epsilon \le -\epsilon$ and thus from (C1) in Definition \ref{Smoothaprox}, $f_\epsilon(x-\epsilon)=f(x-\epsilon)=f_-(x-\epsilon)$. Since supp $\Xi_i =\{x_i \}\subset \RE^-_0$, we get:
$$
\langle
f_\epsilon(x-\epsilon) \cdot \Xi_i , t \rangle =\langle
f_-(x-\epsilon) \cdot \Xi_i , t \rangle =
\langle
  \Xi_i , f_-(x-\epsilon) t \rangle \, .
$$
Moreover, $f_-(x-\epsilon) t \overset{\DO}{\longrightarrow} f_-(x) t$ in the limit $\epsilon \to 0^+$, and thus:
$$
\lim_{\epsilon \downarrow 0} \langle
f_\epsilon(x-\epsilon) \cdot \Xi_i , t \rangle =
\lim_{\epsilon \downarrow 0} \langle
  \Xi_i , f_-(x-\epsilon) t \rangle = 
\langle
  \Xi_i , f_-(x) t \rangle = \langle
  f_- \cdot \Xi_i , t \rangle \, .
$$   
Finally, we also have from (\ref{prodf}) (check (\ref{prods})): 
$$
f * \Xi_i =  \left( H_-f_-+ Hf_+ \right) * \Xi_i= f_- \cdot \Xi_i
$$
and thus:
$$
\lim_{\epsilon \downarrow 0 } \langle
f_\epsilon(x-\epsilon) \cdot \Xi_i , t \rangle =
 \langle  f* \Xi_i, t \rangle \, .
$$
An equivalent result is valid for $x_i >0$. After summing in $i$, we get:
$$
\lim_{\epsilon \downarrow 0} \langle
f_\epsilon(x-\epsilon) \cdot \Xi , t \rangle =
 \langle  f* \Xi , t \rangle \, .
$$ 
Adding this result to (\ref{Caseg-}) we conclude the proof of (\ref{Case-}). An equivalent result can be obtained for the case $\wh F_-$.
	
\end{proof}

\end{section}

\section{Regularity, existence and uniqueness results}\label{g}

In this section we study the general properties of the ODEs with distributional coefficients (\ref{eq0}):  
$$
\sum_{i=0}^n \left(a_i(x) * \psi^{(i)}(x) + \psi^{(i)}(x) * b_i(x)\right)=f(x)  
$$
where $a_i,b_i \in \A$, $f \in \C^\infty$ and $*$ is the distributional product (\ref{prod}, \ref{prodf}).

We also study the associated 
 initial value problems (IVP)s with initial conditions, formally
\begin{equation}\label{IC}
\overline{\psi(x_0)} =\overline{C}
\end{equation}
where $x_0$ is a {\it regular} point of (\ref{eq0}) (cf. Definition 3.1) and
$$
\overline{\psi(x_0)} =(\psi(x_0),...,\psi^{(n-1)}(x_0))^T \quad , \quad \overline{C}=(C_1,...,C_n)^T \in \CO^n
\, .
$$
where the superscript $T$ denotes transposition.

Finally, we will also consider the limit ODEs of the form:
\begin{equation}\label{LimitODE2}
\lim_{\epsilon \downarrow 0} \left(\sum_i \left( a_{i\epsilon}^+ +b_{i\epsilon}^- \right) \cdot \psi^{(i)} \right) =f
\end{equation}
where $\left(a_{i\epsilon}^+\right)_{\epsilon \in I} \in \F_+(a_i)$, $\left( b_{i\epsilon}^-\right)_{\epsilon \in I} \in \F_-(b_i)$ are one-parameter families of 
smooth functions associated with $a_i, b_i \in \A$,  $I=]0,\epsilon_0] $ (cf. Definition \ref{Smoothaprox}).



\subsection{Definitions and preliminary results}

The equations (\ref{eq0}), (\ref{LimitODE2}) and the initial conditions (\ref{IC}) are defined in the distributional sense. More precisely:

\begin{definition}:

\begin{itemize}
\item [(A1)] $\psi$ is a solution of the ODEs (\ref{eq0}) (or (\ref{LimitODE2})) iff $\psi
\in \A$ and $\psi$ satisfies (\ref{eq0}) (respectively (\ref{LimitODE2})) in the sense of distributions. In (\ref{LimitODE2}), the limit $\epsilon \to 0^+$  is taken in $\DO'$. \item [(A2)] A point $x_0 \in \RE$ is said
to be a {\it regular point} of (\ref{eq0}) iff $x_0$ does not belong to the singular supports of $a_i,b_i$. An
interval is a {\it regular interval} of (\ref{eq0}) iff all its points are regular. \item [(A3)] $\psi$
satisfies the initial conditions (\ref{IC}) at a regular point $x_0$ iff there exists an open interval $\Omega \ni x_0$ and a function $\phi \in \C^{\infty}(\Omega)$ such that: (i)
 $\psi=\phi_{\DO'}$ on $\Omega$ (where $\phi_{\DO'}$ denotes the regular distribution associated to
$\phi$), and (ii) $\phi^{(i)}(x_0)= C_{i+1}$, $i=0,...,n-1$.

\end{itemize}
\end{definition}

In the following Corollary of Theorem \ref{Theorem_WL} we show that the eqs.(\ref{eq0}) and (\ref{LimitODE2}) are equivalent for suitable choices of the coefficients.

\begin{corollary} \label{EqLimitODE}
Consider the equation (\ref{eq0}) with coefficients $a_i, b_i \in \A$. Let $I=]0,\epsilon_0]$, and let $\left(a_{i\epsilon}^+\right)_{\epsilon \in I} \in \F_+(a_i)$, $\left(b_{i\epsilon}^-\right)_{\epsilon \in I} \in \F_-(b_i)$ be two one-parameter families of smooth functions associated to $a_i$ and $b_i$, respectively (cf. Definition \ref{Smoothaprox}). Then the limit ODE (\ref{LimitODE2}) with coefficients $a_{i\epsilon}^+,b_{i\epsilon}^-$ is equivalent to eq.(\ref{eq0}). 
	
\end{corollary}

\begin{proof}

In Theorem \ref{Theorem_WL} we have proved that 
\begin{equation} \label{ab_WL}
\mbox{w}\lim_{\epsilon \downarrow 0} a_{i\epsilon}^+\cdot =\wh{a_i}_+ \quad , \quad   
\mbox{w}\lim_{\epsilon \downarrow 0} b_{i\epsilon}^-\cdot =\wh{b_i}_-
\end{equation}
where $a_{i\epsilon}^+\cdot $ and $b_{i\epsilon}^-\cdot $ are defined as operators in $\A$ (of the form (\ref{Smoothop})); and $\wh{a_i}_+\psi = a_i*\psi$ and  
$\wh{b_i}_-\psi = \psi * b_i$ are also operators in $\A$ (of the form (\ref{Aop})).

Let us define the differential operator:
\begin{equation}
\wh{L_\epsilon}=\sum_i \left(a_{i\epsilon}^+ + 
	b_{i\epsilon}^- \right) \cdot D^i_x	
\end{equation}
with domain $\A$. Since $\A$ is closed under differentiation, we have from (\ref{ab_WL}):
$$
\mbox{w}\lim_{\epsilon \downarrow 0} \wh{L_\epsilon} = \sum_i \left(\wh{a_i}_+ + 
	\wh{b_i}_- \right) D^i_x 
$$
Let $\wh L$ denote the operator on the right hand side of the previous equation. Then for every $\psi \in \A$,  we have in the sense of distributions: 
$$
\lim_{\epsilon \downarrow 0} \left(\wh{L_\epsilon} \psi\right) = \wh L \psi
$$
and thus $\psi$ is a solution of (\ref{LimitODE2}) iff it is a solution of $\wh L \psi =f$. The latter equation is exactly eq.(\ref{eq0}). 
 
 \end{proof}

Hence, the equations of the form (\ref{eq0}) provide an alternative formulation for the class of limit ODEs of the form (\ref{LimitODE2}). For the rest of this section we will focus on the properties of the equations (\ref{eq0}); it follows from Corollary \ref{EqLimitODE} that all the results are equally valid for the limit ODEs (\ref{LimitODE2}).

Let us proceed. If $\Omega$ is an open {\it regular} interval of (\ref{eq0}) then the restrictions $a_{i}|_\Omega,b_{i}|_\Omega$
are regular distributions associated to smooth functions. These functions admit a unique smooth extension to
$\overline{\Omega}$ (recall that $a_i,b_i \in \A$, and Theorem \ref{Theorem_A}). Let
$a_{i\overline{\Omega}},b_{i\overline{\Omega}}$ denote these extensions. From now on we will always assume
that (\ref{eq0}) satisfies the following property

\begin{definition} Sectionally Regular ODE. \label{Re}

The ODE (\ref{eq0}) is said to be sectionally regular iff, for every open regular interval $\Omega$ of
(\ref{eq0}), arbitrary $x_0\in \overline{\Omega}$ and $\overline{C}\in \CO^n$, the associated IVP
\begin{equation}\label{CODE}
\sum\limits_{i = 0}^n (a_{i\overline{\Omega}}+b_{i\overline{\Omega}}) \psi^{(i)}_{\overline{\Omega}}=
f|_{\overline{\Omega}} \quad , \quad \overline{\psi_{\overline{\Omega}}(x_0)}=\overline{C}
\end{equation}
has a unique solution $\psi_{\overline{\Omega}} \in \C^{\infty}(\overline{\Omega})$.

\end{definition}

The next theorem provides sufficient conditions for (\ref{eq0}) to be sectionally regular:

\begin{theorem} \label{TheoSSR}
Consider the ODE (\ref{eq0}) with coefficients $a_i,b_i \in \A$ , $i=0,..,n$ such that, for every regular
interval ${\Omega}$,
\begin{equation} \label{CondSR}
a_{n\overline{\Omega}}(x)+b_{n\overline{\Omega}}(x) \not=0 \quad , \quad \forall x \in \overline{\Omega} \, .
\end{equation}
Then (\ref{eq0}) is sectionally regular.
\end{theorem}

\begin{proof}
In view of (\ref{CondSR}), for every regular interval ${\Omega}$, we can re-write (\ref{CODE}) in the form
\begin{equation}\label{CODE3}
\psi^{(n)}_{\overline\Omega}=\sum\limits_{i = 0}^{n-1} c_{i} \psi^{(i)}_{\overline\Omega} +
f|_{\overline{\Omega}} \quad , \quad \overline{\psi_{\overline{\Omega}}(x_0)}=\overline{C}
\end{equation}
where
$$
c_i= - \frac{a_{i\overline{\Omega}}+b_{i\overline{\Omega}}}{a_{n\overline{\Omega}}+b_{n\overline{\Omega}}} \in
C^{\infty}(\overline{\Omega}) \quad , \quad i=0,..,n-1 \, .
$$
It follows from Whitney extension theorem \cite{Whi34} that the functions $c_i$ admit a smooth extension to an
open interval $I \supset \overline{\Omega}$. Let $c_{iI}$, $i=0,..,n-1$ denote these extensions. Picard's
theorem then implies that the solution of the linear IVP
\begin{equation} \label{CODE4}
\psi^{(n)}_I=\sum\limits_{i = 0}^{n-1} c_{iI} \psi^{(i)}_I + f|_{I} \quad , \quad  \overline{\psi_I(x_0)}=
\overline{C}
\end{equation}
exists and is unique for each initial conditions given at $x_0 \in I$.

Since $c_{iI}, f|_{I} \in \C^{\infty}(I)$, it is also well-known from the theory of linear ODEs that $\psi_I$ is
maximal defined on the whole interval $I$ and is smooth (cf. [Lemma 2.3 and Theorem 3.9 \cite{Teschl}]). Hence (\ref{CODE3}), and thus (\ref{CODE}), have a unique smooth solution, and so (\ref{eq0}) is sectionally regular.

\end{proof}

Finally, in the next theorem we prove a simple result relating the solutions of (\ref{CODE}) and (\ref{eq0}); it will be used in the next section.

\begin{theorem} \label{SolSR}

Let (\ref{eq0}) be sectionally regular, and let $\Omega$ be an open regular interval of (\ref{eq0}). If $\psi$
is a solution of (\ref{eq0}) then, on $\Omega$, it satisfies $\psi= \psi_{\overline{\Omega}}$ for
$\psi_{\overline{\Omega}}$ a solution of (\ref{CODE}) for some initial data.

\end{theorem}

\begin{proof}

If $\Omega$ is an open regular interval of (\ref{eq0}) then on $\Omega$ (\ref{eq0}) reduces to:
\begin{equation} \label{CODE2}
\sum\limits_{i = 0}^n (a_{i}|_{\Omega}+b_{i}|_{\Omega}) \psi|_{{\Omega}}^{(i)}= f|_{{\Omega}} \, ,
\end{equation}
which is a consequence of $a_{i}|_{\Omega},b_{i}|_{\Omega} \in \C^{\infty}({\Omega})$, $\psi^{(i)} \in \A$ and eq.(\ref{2.6E}). 



Since (\ref{CODE2}) is a restriction of (\ref{CODE}) to $\Omega$, its solutions are restrictions (to $\Omega$) of the solutions of (\ref{CODE}). Let $\psi$ be a solution of (\ref{eq0}). Then it satisfies (\ref{CODE2}) on $\Omega$, and so there exists $\psi_{\overline{\Omega}}$, solution of (\ref{CODE}) for some initial data, such that $\psi=\psi_{\overline{\Omega}}$ on $\Omega$.

\end{proof}

\subsection{Main Results}

We assume, to simplify the discussion, that sing supp $a_i,b_i \subseteq\{0\}$, $i=0,..,n$. If there are more
(but a finite number of) singular points, the main results are essentially the same. The coefficients of
(\ref{eq0}) can then be written:
\begin{eqnarray}\label{Coeff}
a_i & = & H_- a_{i-}+H a_{i+} + A_i \quad,  \quad  i=0,..,n\nonumber \\
b_i& = & H_- b_{i-}+H b_{i+} + B_i \quad , \quad i=0,..,n
\end{eqnarray}
where $a_{i\pm},  b_{i\pm} \in \C^\infty (\RE) $, supp $A_i$, supp $B_i \subseteq \{0\}$, $i=0,..,n$. Hence,
both $A_i$ and $B_i$ are given by a finite linear combination of a Dirac delta and its derivatives:
\begin{equation} \label{Coef}
A_i(x) = \sum_k A_{ik} \delta^{(k)}(x) \quad , \quad B_i(x) = \sum_k B_{ik} \delta^{(k)}(x) \quad , \quad
i=0,..,n
\end{equation}
where $A_{ik},B_{ik} \in \CO$ and $k\in \N_0$. Finally, let
$$
M= {\rm max} \, \{{\rm ord}\,  A_i \, , \,  {\rm ord} \, B_i \, ; \, i =0,...,n\} \, .
$$
Then

\begin{theorem}\label{1i}
Consider the ODE (\ref{eq0}) with coefficients of the form (\ref{Coeff},\ref{Coef}) and satisfying
$a_{n-}(0)+b_{n+}(0)\not=0$. If (\ref{eq0}) is sectionally regular then its general solution is of the form
\begin{equation} \label{GF}
\psi=H_- \psi_-+H\psi_+ + \Delta
\end{equation}
where $\psi_- , \psi_+ \in \C^\infty (\RE)$ satisfy
\begin{equation}\label{R-}
\sum_{i=0}^n(a_{i-}+b_{i-}) \psi^{(i)}_- =f \quad \mbox{on} \quad \RE_0^-
\end{equation}
and
\begin{equation} \label{R+}
\sum_{i=0}^n(a_{i+}+b_{i+}) \psi^{(i)}_+ =f \quad \mbox{on} \quad \RE_0^+
\end{equation}
respectively. Moreover, $\Delta \in \DO'$ satisfies:

(i) if $M \le n$ then $\Delta=0$.

(ii) if $M > n$ then supp $\Delta \subseteq \{0\}$ and ord $\Delta \le M-n$.

\end{theorem}

\begin{proof}

Substituting (\ref{Coeff}) in (\ref{eq0}), we get:
\begin{eqnarray}\label{i}
&& \sum\limits_{i = 0}^n \left(\left(H_- a_{i-}+H a_{i+} \right) * \psi^{(i)} +\psi^{(i)} *\left(H_- b_{i-}+H
b_{i+} \right)\right)\\
&+& \sum\limits_{i = 0}^n\left( A_i * \psi^{(i)} + \psi^{(i)}* B_i\right)= f \, . \nonumber
\end{eqnarray}
Both $\RE^-$ and $\RE^+$ are regular intervals of (\ref{i}), and since (\ref{i}) is sectionally regular, the
equations (\ref{R-}) and (\ref{R+}) have unique smooth solutions for initial data given at $x_0 \in
\RE^-_0$ and $x_0 \in \RE^+_0$, respectively (cf. Definition \ref{Re}). In view of Whitney's extension theorem
\cite{Whi34,Fle77}, these solutions admit smooth extensions to $\RE$. Denote these extensions by $\psi_-$ and
$\psi_+$, respectively.

From Theorem \ref{SolSR}, we conclude that if $\psi$ satisfies (\ref{i}) then necessarily $\psi=\psi_-$ on
$\RE^-$ and $\psi=\psi_+$ on $\RE^+$, for some $\psi_-,\psi_+ \in C^{\infty}(\RE)$ satisfying (\ref{R-}) and
(\ref{R+}), respectively. Hence:
\begin{equation}\label{ii}
 \psi = H_-\psi_- +H \psi_+ + \Delta
\end{equation}
where $\textrm{supp } \Delta \subseteq\left\{0\right\} $. This proves the statement of
eqs.(\ref{GF},\ref{R-},\ref{R+}).

To proceed, let us calculate $\psi^{(i)}$ from (\ref{ii}):
\begin{equation}\label{psii}
\psi^{(i)}=H_-\psi_-^{(i)} +H \psi_+^{(i)} + \Delta^{(i)} + \sum_{j=1}^i \left( \begin{gathered}
  i \hfill \\
  j \hfill \\
\end{gathered} \right) \delta^{(j-1)} \left(\psi_+^{(i-j)}- \psi_-^{(i-j)}    \right) ,\quad i \ge 1
\end{equation}
Substituting (\ref{ii}) and (\ref{psii}) into (\ref{i}) and taking into account (\ref{prods}), we get:
\begin{eqnarray} \label{Y}
& & \sum_{i=0}^n \left(H_-\left(a_{i-}+b_{i-}\right)\psi_-^{(i)} + H \left(a_{i+}+b_{i+}\right)\psi_+^{(i)}\right) \\
& & + \sum_{i=0}^n \left(A_i \psi_+^{(i)} + B_i \psi_-^{(i)}\right) + \sum_{i=0}^n  \left(a_{i-} + b_{i+}\right) \Delta^{(i)}
\nonumber \\
& & +\sum_{i=1}^n \left(a_{i-}+b_{i+}\right) \sum_{j=1}^i \left( \begin{gathered}
  i \hfill \\
  j \hfill \\
\end{gathered} \right) \delta^{(j-1)} \left(\psi_+^{(i-j)}- \psi_-^{(i-j)}    \right) =f \nonumber
\end{eqnarray}
Using (\ref{R-}) and (\ref{R+}), the first term cancels the right hand side, and so:
\begin{equation}\label{imm}
(a_{0-}+b_{0+})\Delta +\sum\limits_{i = 1}^n \left(a_{i-}+b_{i+}\right) \left(\Delta^{(i)}+\Gamma_i\right)+
\sum\limits_{i = 0}^n \left(A_i\psi_+^{(i)}+B_i\psi_-^{(i)}\right) =0
\end{equation}
where
\begin{equation} \label{Omegai}
\Gamma_i=\sum\limits_{j=1}^i \left( \begin{gathered}
  i \hfill \\
  j \hfill \\
\end{gathered}  \right) \, \delta^{(j-1)} \left( \psi_+^{(i-j)}-\psi_-^{(i-j)}\right) \quad , \quad i=1,..,n
\end{equation}
Hence, ord $\Gamma_i \le i$ for all $i=1,..,n$ and so
$$
{\rm max} \left\{ {\rm ord} \, \Gamma_i \, , \, i=1,..,n\right\} \le n \, .
$$
On the other hand, supp $\Delta \subseteq\{0\}$ and so $\Delta=0$ or ord $\Delta \ge 1$. In the latter case, ord
$\Delta^{(n)} \ge n+1$ and also ord $(a_{n-}+b_{n+}) \Delta^{(n)} \ge n+1$ (recall that, by assumption,
$a_{n-}(0)+b_{n+}(0) \not=0$).

Since the terms of different orders in eq.(\ref{imm}) are linearly independent, if $M \le n$ and $\Delta \not=0$
then $(a_{n-}+b_{n+})\Delta^{(n)}$ cannot be cancelled by any other term in (\ref{imm}). Hence, we must have
$\Delta=0$.

On the other hand, if $M > n$ then we must have ord $\Delta^{(n)} \le M$, in which case ord $\Delta \le M-n$,
concluding the proof.

\end{proof}

\begin{remark} \label{remark}

If $a_{n-}(0)+b_{n+}(0)=0$, the situation is more involved, but the conclusions also follow from eq.(\ref{imm}), and the analysis is basically the same. In this case the order of the solutions of
(\ref{eq0}) depends on the properties of the lower order coefficients $a_{i-}+b_{i+}$, $i=0,..,n-1$, and on the derivatives of $a_{n-}+b_{n+}$ at $x=0$.

\end{remark}

\begin{theorem}\label{1ii}

Consider the ODE (\ref{eq0}) with coefficients of the form (\ref{Coeff},\ref{Coef}) and satisfying
$a_{n-}(0)+b_{n+}(0) \not=0$, $a_{n\pm}(0)+b_{n\pm}(0) \not=0$. If (\ref{eq0}) is sectionally regular and $M \le
n$ then every solution is of the form
\begin{equation} \label{3.15}
\psi=H_- \psi_-+H\psi_+
\end{equation}
where $\psi_-,\psi_+ \in \C^\infty$ are solutions of (\ref{R-},\ref{R+}). Moreover, $\psi$ satisfies an
interface condition
\begin{equation}\label{Feq}
\widehat F \psi=0
\end{equation}
where $\widehat F$ is a singular operator of rank, at most, $n$:
\begin{equation}\label{Ffinal}
\widehat F \psi= \sum_{i=0}^{n-1} \delta^{(i)}(x) f_i\left(\psi_\pm(0),...,\psi_\pm^{(n-1)}(0)\right)
\end{equation}
and the functions $f_i:\CO^{2n} \to \CO$ are linear.
\end{theorem}

\begin{proof}
If $M \le n$, it follows from Theorem \ref{1i}(i) that $\Delta =0$. Hence, the general solution of (\ref{eq0})
is of the form
$$
\psi= H_-\psi_-+H\psi_+
$$
where $\psi_-,\psi_+$ satisfy (\ref{R-},\ref{R+}).

It also follows from (\ref{imm}) that
\begin{equation} \label{F1}
\widehat F \psi := \sum\limits_{i=1}^n (a_{i-} + b_{i+}) \Gamma_i + \sum\limits_{i=0}^n \left(A_i \psi^{(i)}_+ +
B_i \psi^{(i)}_-\right)=0
\end{equation}
where
\begin{eqnarray} \label{Omega}
\Gamma_i &=& \sum\limits_{j=1}^i \left( \begin{gathered}
  i \hfill \\
  j \hfill \\
\end{gathered}  \right) \, \delta^{(j-1)} (x)\left[ \psi_+^{(i-j)}(x)-\psi_-^{(i-j)}(x)\right]  \\
&=& \sum\limits_{j=1}^i \, \delta^{(j-1)} (x)\left[ \psi_+^{(i-j)}(0)-\psi_-^{(i-j)}(0)\right] , \quad i=1,..,n
\nonumber
\end{eqnarray}
In order to prove the second identity in (\ref{Omega}), we proceed by induction. The identity is trivial for
$i=1$. Moreover
\begin{eqnarray}
\Gamma_{i+1} &=& \sum\limits_{j=1}^{i+1} \left( \begin{gathered}
  i+1 \\
  j  \\
\end{gathered}  \right) \, \delta^{(j-1)} (x)\left[ \psi_+^{(i+1-j)}(x)-\psi_-^{(i+1-j)}(x)\right] \nonumber \\
&=& D^{i+1}_x\left[ H\psi_+(x)+H_-\psi_-(x)\right]-\left[H\psi_+^{(i+1)}(x)+H_-\psi_-^{(i+1)}(x)\right] \nonumber \\
&=& D_x \left[H\psi_+^{(i)}(x)+H_-\psi_-^{(i)}(x) + \Gamma_i\right] -\left[H\psi_+^{(i+1)}(x)+H_-\psi_-^{(i+1)}(x)\right] \nonumber \\
&=& \delta(x) \left[\psi_+^{(i)}(0)-\psi_-^{(i)}(0)\right] + D_x(\Gamma_i) \nonumber
\end{eqnarray}
and so, assuming that (\ref{Omega}) is valid for $i$, we get
\begin{eqnarray}
\Gamma_{i+1} &=& \delta(x) \left[\psi_+^{(i)}(0)-\psi_-^{(i)}(0)\right] + \sum\limits_{j=1}^i \,
\delta^{(j+1-1)}
(x)\left[ \psi_+^{(i-j)}(0)-\psi_-^{(i-j)}(0)\right] \nonumber \\
&=& \sum\limits_{j=1}^{i+1} \, \delta^{(j-1)} (x)\left[ \psi_+^{(i+1-j)}(0)-\psi_-^{(i+1-j)}(0)\right] \, .
\nonumber
\end{eqnarray}
Hence, (\ref{Omega}) is valid for all $i \in \N$.

To proceed, we consider the second sum in (\ref{F1}). Since $M \le n$, all the coefficients $A_i,B_i$ are of the
form
$$
A_i= \sum_{k=0}^{n-1} A_{ik} \delta^{(k)}(x) \quad , \quad B_i= \sum_{k=0}^{n-1} B_{ik} \delta^{(k)}(x)
  , \quad A_{ik},B_{ik} \in \CO  , \quad i=0,..,n
$$
and so
\begin{eqnarray} \label{Ai}
A_i (x) \psi_+^{(i)}(x) &=& \sum\limits_{k=0}^{n-1} A_{ik} \delta^{(k)} (x) \psi_+^{(i)}(x)  \nonumber \\
& = &\sum\limits_{k=0}^{n-1} A_{ik} \sum\limits_{j=0}^{k} \left( \begin{gathered}
  k \hfill \\
  j \hfill \\
\end{gathered}  \right) \, (-1)^{j+k} \delta^{(k-j)} (x) \psi_+^{(i+j)}(0) \\
&=&\sum\limits_{j \le k=0}^{n-1} \left( \begin{gathered}
  k \hfill \\
  j \hfill \\
\end{gathered}  \right) \, (-1)^{j+k}  A_{ik} \delta^{(k-j)} (x) \psi_+^{(i+j)}(0) \nonumber
\end{eqnarray}
and likewise
\begin{equation} \label{Bi}
B_i (x) \psi_-^{(i)}(x) = \sum\limits_{j \le k=0}^{n-1} \left( \begin{gathered}
  k \hfill \\
  j \hfill \\
\end{gathered}  \right) \, (-1)^{j+k}  B_{ik} \delta^{(k-j)} (x) \psi_-^{(i+j)}(0)
\end{equation}
It follows from (\ref{Omega},\ref{Ai},\ref{Bi}) that the terms in (\ref{F1}) satisfy:
$$
{\rm supp} \, \Gamma_i \, , \, {\rm supp} \, (A_i \psi_+^{(i)}) \, , \, {\rm supp} \, (B_i \psi_-^{(i)})
\subseteq \{0\}
$$
$$
{\rm ord}\,  \Gamma_i \, , \, {\rm ord} \, (A_i \psi_+^{(i)}) \, , \, {\rm ord}\,  (B_i \psi_-^{(i)}) \le n
$$
and that $\widehat F \psi(x)$ is linear (and exclusively) dependent on the entries
$$
\psi_\pm(0),...,\psi_\pm^{(n-1)}(0),\psi_\pm^{(n)}(0),...,\psi_\pm^{(2n-1)}(0) \, .
$$
Hence, $\widehat F$ is a rank
$n$ linear operator of the form
\begin{equation} \label{F2}
\widehat F \psi=\sum_{i=0}^{n-1} \delta^{(i)}(x) \widetilde
f_i(\psi_\pm(0),...,\psi_\pm^{(n-1)}(0),\psi_\pm^{(n)}(0),...,\psi_\pm^{(2n-1)}(0))
\end{equation}
where $\widetilde f_i:\CO^{4n} \longrightarrow \CO$ are linear functions.

Finally, the equations (\ref{R-}) and (\ref{R+}) can be used to express $\psi_\pm^{(n)}(0)$ (and all its
derivatives up to $\psi_\pm^{(2n-1)}(0)$) in terms of $\psi_\pm(0),...,\psi_\pm^{(n-1)}(0)$. This is possible
because $a_{n\pm}(0)+b_{n\pm}(0) \not=0$. Moreover, these relations are linear (because eqs.(\ref{R-},\ref{R+})
are linear). It follows that $\widehat F \psi$ can be re-written in the form
$$
\widehat F \psi(x)=\sum_{i=0}^{n-1} \delta^{(i)}(x) f_i(\psi_\pm(0),...,\psi_\pm^{(n-1)}(0))
$$
where $f_i:\CO^{2n} \longrightarrow \CO$ are linear functions, which concludes the proof.
\end{proof}

The interface conditions (\ref{Feq}, \ref{Ffinal}) can be written in the form:
\begin{equation} \label{FC}
f_i(\psi_\pm(0),...,\psi_\pm^{(n-1)}(0))=0 \quad , \quad i=1,...,n
\end{equation}
yielding a system of $n$ linear equations for the $2n$ unknowns
$$
\overline{\psi_-(0)}=(\psi_-(0),...,\psi^{(n-1)}_-(0))^T \quad , \quad 
\overline{\psi_+(0)}=(\psi_+(0),...,\psi^{(n-1)}_+(0))^T \, .
$$
The equations (\ref{FC}) can be re-written as:
\begin{equation} \label{Meq}
{\bf A} \overline{\psi_-(0)} = {\bf B} \overline{\psi_+(0)}
\end{equation}
where $\bf A,\bf B$ are $n \times n$ complex valued matrices. We will use this form of the interface conditions to study several of its properties. 

The conditions (\ref{Meq}) can be {\it separating} or {\it interacting}. In the separating case they reduce to a
set of conditions for $\overline{\psi_+(0)}$ and another set of conditions for $\overline{\psi_-(0)}$. In this
case the values of $\overline{\psi_+(0)}$ and $\overline{\psi_-(0)}$ are independent of each other. In the
interacting case the conditions relate the values of $\overline{\psi_+(0)}$ with those of
$\overline{\psi_-(0)}$. If they do not completely fix the values of $\overline{\psi_-(0)}$ in terms of those of
$\overline{\psi_+(0)}$ or vice-versa, we say that the conditions are only {\it partially interacting}.

Let, as usual, $\mbox{Ker} \,{\bf X}$ and  $\mbox{Ran} \,{\bf X}$ denote the kernel and the range of the matrix
$\bf X$. Then

\begin{theorem} \label{1iii}
Consider the interface conditions (\ref{Meq}) and let $\W= \mbox{Ran} \,{\bf A} \, \cap \mbox{Ran} \,{\bf B}$.
\begin{enumerate}
\item [(1)] If $\W=\{0 \}$ then the conditions (\ref{Meq}) are separating.

\item [(2)] If $\W=\CO^n$ then the conditions (\ref{Meq}) are interacting.

\item [(3)] If $\W \not=\{ 0 \}$ and $\W \not= \CO^n$ then the conditions (\ref{Meq}) are {\it partially
interacting}.
\end{enumerate}

\end{theorem}

\begin{proof}

(1) If $\W =\{0\}$ then
$$
{\bf A} \overline{\psi_-(0)} = {\bf B} \overline{\psi_+(0)} \Longleftrightarrow \left\{
\begin{array}{l}
{\bf A} \overline{\psi_-(0)} = 0 \\
\\
{\bf B} \overline{\psi_+(0)} =0
\end{array} \right.
$$
and the conditions are separating.

(2) If $\W=\CO^n$ then  $\mbox{Ran} \,{\bf A} = \, \mbox{Ran} \, {\bf {B}}$ $= \CO^n \Longleftrightarrow
\mbox{Ker} \,{\bf A} = \mbox{Ker} \, {\bf B} =\{0\}$. Hence, both ${\bf A}$ and ${\bf B}$ are invertible. It
follows that the values of $\overline{\psi_-(0)}$ and $\overline{\psi_+(0)}$ are completely and uniquely
determined from each other:
$$
{\bf A} \overline{\psi_-(0)} = {\bf B} \overline{\psi_+(0)} \Longleftrightarrow  \overline{\psi_-(0)} ={\bf
A}^{-1} {\bf B} \overline{\psi_+(0)} \Longleftrightarrow  \overline{\psi_+(0)} ={\bf B}^{-1} {\bf A}
\overline{\psi_-(0)}.
$$

(3) If $\W \not=\CO^n$ then either $\mbox{Ran} \, {\bf A} $ or  $\mbox{Ran} \,{\bf B}$ (or both) is not $\CO^n$.
Assume that $\mbox{Ran} \,{\bf A} \not= \CO^n$. Then  $\mbox{Ker} \,{\bf A} \not=\{0\}$ and if
$(\overline{\psi_-(0)},\overline{\psi_+(0)})$ is a solution of (\ref{Meq}) then
$(\overline{\psi_-(0)}+\overline{\xi_-(0)},\overline{\psi_+(0)})$ is also a solution for all
$\overline{\xi_-(0)} \in  \mbox{Ker} \, {\bf A}$. Hence, the values of $\overline{\psi_-(0)}$ are not completely
fixed by those of $\overline{\psi_+(0)}$ (and vice-versa, if  $\mbox{Ker} \,{\bf B} \not=\{0\}$).

On the other hand, since $\W \not=\{0\}$ and $\W$ is a linear space, there are two different vectors
$\overline{X}, \overline{Y} \in \W$ and two different pairs $(\overline{\psi_-(0)},\overline{\psi_+(0)})$ and
$(\overline{\xi_-(0)},\overline{\xi_+(0)})$ such that
$$
\left\{
\begin{array}{l}
{\bf A} \overline{\psi_-(0)} = \overline{X} \\
\\
{\bf B} \overline{\psi_+(0)} =\overline{X}
\end{array} \right. \quad , \quad
\left\{
\begin{array}{l}
{\bf A} \overline{\xi_-(0)} = \overline{Y} \\
\\
{\bf B} \overline{\xi_+(0)} =\overline{Y}
\end{array} \right.
$$
Then, of course, $(\overline{\psi_-(0)},\overline{\psi_+(0)})$ and $(\overline{\xi_-(0)},\overline{\xi_+(0)})$
are solutions of (\ref{Meq}), but $(\overline{\psi_-(0)},\overline{\xi_+(0)})$ and
$(\overline{\xi_-(0)},\overline{\psi_+(0)})$ are not. Hence, the values of $\overline{\psi_-(0)}$ and
$\overline{\psi_+(0)}$ that solve (\ref{Meq}) are not completely uncorrelated. However, since they are not
completely fixed by each other, the conditions (\ref{Meq}) are only partially interacting.

\end{proof}

The conditions $\widehat F \psi(x)=0 \Longleftrightarrow {\bf A} \overline{\psi_-(0)} = {\bf B} \overline{\psi_+(0)}$ also determine whether the solution of the IVP (\ref{eq0},\ref{IC}) exists
and is unique. Let
\begin{eqnarray}
\K_ {{\bf A}} & = & \{\overline{X} \in \CO^n : \, {\bf A} \overline{X} \in {\rm \mbox{Ran}} \, {\bf B}\}
\nonumber \\
\K_{\bf B} & = & \{\overline{X} \in \CO^n : \, {\bf B} \overline{X} \in {\rm \mbox{Ran}} \, {\bf A}\} \, .
\end{eqnarray}
We have, of course
$$
\rm \mbox{Ker} \, {\bf A} \subseteq \K_{\bf A} \quad , \quad \rm \mbox{Ker} \, {\bf B} \subseteq \K_{\bf B} \, .
$$
Then

\begin{theorem} \label{1iv}
Consider a sectionally regular ODE (\ref{eq0}) with coefficients (\ref{Coeff},\ref{Coef}) satisfying
$a_{n-}(0)+b_{n+}(0) \not=0$, $a_{n\pm}(0)+b_{n\pm}(0) \not=0$ and $M\le n$. In view of Theorem \ref{1ii}, its solutions satisfy the interface conditions (\ref{Feq},  \ref{Ffinal}). Let us write these conditions in the form (\ref{Meq}).

Consider also the initial conditions:
\begin{equation}\label{IC-}
\overline{\psi(x_0)}=\overline{C} \quad , \quad x_0 < 0 \, .
\end{equation}
and let $\psi_-$ be the solution of the {\it associated} IVP (\ref{R-}) that satisfies
$\overline{\psi_-(x_0)}=\overline{C}$. Then:
\begin{enumerate}
\item [(1)] If $\overline{\psi_-(0)} \notin \K_{\bf A}$ then the IVP (\ref{eq0},\ref{IC-}) has no solutions.

\item [(2)] If $\overline{\psi_-(0)} \in \K_{\bf A}$ then the solutions of the IVP (\ref{eq0},\ref{IC-}) form an
affine space of dimension dim ($\mbox{Ker} \,{\bf B}$).

\item [(3)] The solution of the IVP (\ref{eq0},\ref{IC-}) exists and is unique for arbitrary initial conditions
given at $x_0 < 0$ iff  $\mbox{Ker} \,{\bf B} = \{0\}$.

\end{enumerate}

Equivalent statements can be made for the IVP (\ref{eq0}) with initial conditions
\begin{equation}\label{IC+}
\overline{\psi(x_0)}=\overline{C} \quad , \quad x_0 > 0 \, .
\end{equation}
In this case, let $\psi_+$ be the solution of the {\it associated} IVP (\ref{R+}) that satisfies
$\overline{\psi_+(x_0)}=\overline{C}$. Then:

\begin{enumerate}

\item [(4)] If $\overline{\psi_+(0)} \notin \K_{\bf B}$ then the IVP (\ref{eq0},\ref{IC+}) has no solutions.

\item [(5)] If $\overline{\psi_+(0)} \in \K_{\bf B}$ then the solutions of the IVP (\ref{eq0},\ref{IC+}) form an
affine space of dimension dim ($\mbox{Ker} \,{\bf A}$).

\item [(6)] The solution of the IVP (\ref{eq0},\ref{IC+}) exists and is unique for all initial conditions at
$x_0 > 0$ iff  $\mbox{Ker} \,{\bf A} = \{0\}$.

\end{enumerate}

\end{theorem}

\begin{proof}
Recall from Theorem \ref{1ii} that if the conditions of Theorem \ref{1iv} hold, and $\psi$ is a solution of
(\ref{eq0}) then it satisfies (\ref{3.15},\ref{R-},\ref{R+}) and the interface condition (\ref{Meq}). Moreover,
the {\it associated} IVPs (\ref{R-},\ref{R+}) have unique smooth solutions for arbitrary initial conditions
given at $x_0 \le 0$ and $x_0 \ge 0$, respectively. Then:

(1) If $\overline{\psi_-(0)} \notin \K_{\bf A}$ then ${\bf A} \overline{\psi_-(0)} \notin  \mbox{Ran} \, {\bf
B}$ and so
$$
\not\exists \, \overline{\psi_+(0)} \in \CO^n: \quad {\bf B} \overline{\psi_+(0)} = {\bf A}
\overline{\psi_-(0)}\, .
$$
Hence, the condition (\ref{Meq}) doesn't have a solution and so the IVP (\ref{eq0},\ref{IC-}) has no solutions satisfying
the initial condition $\overline{\psi(x_0)}= \overline{C}$.

(2) If $\overline{\psi_-(0)} \in \K_{\bf A}$ then ${\bf A} \overline{\psi_-(0)} \in  \mbox{Ran} \,{\bf B}$ and
so
\begin{equation}\label{Xeq}
\exists \, \overline{X} \in \CO^n: \quad {\bf B}\overline{X} = {\bf A} \overline{\psi_-(0)} \, .
\end{equation}
Let $\psi_+$ be the solution of the IVP (\ref{R+}) with initial condition $\overline{\psi_+(0)}=\overline{X}$.
Then
$$
\psi= H_- \psi_- + H \psi_+
$$
is a global solution of the IVP (\ref{eq0},\ref{IC-}).

Let now $\overline{Y} \in  \mbox{Ker} \, {\bf B}$. Then
$\overline{X}+\overline{Y}$ is also a solution of (\ref{Xeq}). Let
${\xi_+}$ be the solution of the IVP (\ref{R+}) with $f=0$ that
satisfies $\overline{\xi_+(0)}=\overline{Y}$. Then
$$
\tilde\psi= H_- \psi_- + H (\psi_++\xi_+)
$$
is also a global solution of the IVP (\ref{eq0},\ref{IC-}). Hence, the dimension of the affine space of
solutions of the IVP (\ref{eq0},\ref{IC-}) is dim ($\mbox{Ker} \,{\bf B}$).

(3) It follows from the previous point (2) that the IVP (\ref{eq0},\ref{IC-}) has solutions for all initial
conditions $\overline{\psi(x_0)}=\overline{C}$ given at $x_0 < 0$ iff $\K_{\bf A} = \CO^n$. Moreover, the
solution is unique iff  $\mbox{Ker} \,{\bf B}=\{0\}$. It turns out that the latter condition also implies the
former one:
$$
{\rm \mbox{Ker}}\, {\bf B}=\{0\} \Longrightarrow {\rm \mbox{Ran}}\, {\bf B}=\CO^n \Longrightarrow {\bf A}
\overline{X}\in {\rm \mbox{Ran}}\, {\bf B} , \, \forall \overline{X} \in \CO^n \Longrightarrow \K_{\bf A}=\CO^n
$$
Hence, the solution of the IVP (\ref{eq0},\ref{IC-}) exists and is unique for all initial data given at $x_0 <0$
iff ${\rm \mbox{Ker}} {\bf B}=\{0\}$.

(4), (5) and (6): The proof is identical to that of (1), (2) and (3).

\end{proof}

\begin{corollary}\label{Corollary-1iv}

If the conditions (3) and (6) in the previous Theorem are both satisfied, then the interface conditions (\ref{Meq}) are interacting, and vice-versa. Hence
interacting interface conditions (\ref{Meq}) are necessary and sufficient for the existence and uniqueness of
the solutions of the IVP (\ref{eq0}) with initial conditions given at an arbitrary $x_0 \in \RE\backslash \{0\}$.

\end{corollary}

\begin{remark} 
 
 The condition (\ref{Meq}) is, in general, asymmetric. For instance, we may have ${\rm \mbox{Ker} \,}
{\bf A}=\{0\}$ but ${\rm \mbox{Ker} \,} {\bf B} \not=\{0\}$, in which case $\overline{\psi_-(0)}$ is completely
determined by $\overline{\psi_+(0)}$ but not vice-versa. Consequently, in this case, the IVP (\ref{eq0}) has
unique solutions for arbitrary initial conditions given at $x_0 >0$, but not for initial conditions given at
$x_0<0$. In the latter case, a solution may not exist, and if exists, it may not be unique (see the example in the next section).

\end{remark}

\section{Simple Example}

In this section we study the ODE with singular coefficients
\begin{equation}\label{ODE2}
\psi''+(k^2 + \alpha \delta''') * \psi =0 \quad , \quad k \in \RE \backslash \{0\}
\end{equation}
in order to illustrate some of the results stated in Theorems \ref{1i}, \ref{1iii} and \ref{1iv}. Here
$\alpha$ and $k$ are real parameters and $\delta'''$ denotes $\delta^{(3)}(x)$. Hence $n=2$ and (if $\alpha
\not=0$) $M=4$. Notice that this ODE does not satisfy all the conditions of Theorems \ref{1iii} and \ref{1iv} (because $M > n$). However, we will see below that the interface condition for (\ref{ODE2}) can still be written in the form (\ref{Feq},\ref{Ffinal}), and thus the results of Theorems \ref{1iii} and \ref{1iv} are still valid for (\ref{ODE2}).  

Let us proceed. On $\RE^-$ and $\RE^+$, (\ref{ODE2}) reduces to
\begin{equation}\label{ODE1}
\psi''+k^2  \psi =0 \, .
\end{equation}
Hence (\ref{ODE2}) is sectionally regular (cf. Definition \ref{Re}) and its general solution can be written in
the form
\begin{equation} \label{SODE2}
\psi =H_- \psi_- + H \psi_+ + \Delta
\end{equation}
where supp $\Delta \subseteq \{0\}$ and $\psi_-, \psi_+$ satisfy (\ref{ODE1}) on $\RE$ (it is not necessary that
$\psi_\pm$ are defined on $\RE$, but it simplifies the presentation). Substituting (\ref{SODE2}) in
(\ref{ODE2}) we get:
\begin{equation}\label{4.4}
 H_-\left[ \psi''_-+k^2\psi_-\right]+ H\left[ \psi''_++k^2\psi_+\right] + \Delta''+k^2 \Delta + \Gamma_2 +
\alpha  \delta''' \psi_+ =0
\end{equation}
where
\begin{eqnarray}\label{Ome}
\Gamma_2 &=&  2  \left(\psi'_+(x)-\psi_-'(x)\right) \delta + \left(\psi_+(x)-\psi_-(x)\right) \delta' \nonumber
\\
&=& \left(\psi'_+(0)-\psi_-'(0)\right)\delta + \left(\psi_+(0)-\psi_-(0)\right)\delta'
\end{eqnarray}
in accordance with (\ref{Omegai}).
Taking into account that
\begin{equation}\label{Omegaa}
\delta''' \psi_\pm(x) = \delta''' \psi_\pm(0) -3 \delta'' \psi'_\pm(0) + 3\delta' \psi_\pm''(0)-\delta \psi_\pm'''(0)
\end{equation}
and separating the terms that depend on the delta distribution from those that do not, we get from (\ref{4.4}):
\begin{equation} \label{ODE+-}
\psi''_-+k^2\psi_- =0 \qquad  , \qquad  \psi''_++k^2\psi_+=0
\end{equation}
and
\begin{eqnarray} \label{ODE3}
&& \Delta''+k^2 \Delta + \delta \left[ \psi_+'(0)-\psi_-'(0) - \alpha \psi_+'''(0)  \right] \nonumber \\
&+& \delta' \left[ \psi_+(0)-\psi_-(0) +3 \alpha \psi_+''(0)  \right] \nonumber \\
&+& \delta'' \left[ - 3\alpha \psi_+'(0)  \right] + \delta''' \left[ \alpha \psi_+(0) \right]=0
\end{eqnarray}
The terms of order higher than two yield
\begin{eqnarray}\label{h}
&& \Delta''= -\delta'' \left[ - 3\alpha \psi_+'(0)  \right] - \delta''' \left[ \alpha \psi_+(0) \right]
\nonumber \\
&\Longrightarrow &  \Delta = -\delta \left[- 3\alpha \psi_+'(0)  \right] - \delta' \left[ \alpha \psi_+(0)
\right] \, .
\end{eqnarray}
Substituting $\Delta$ and $\Delta''$ into (\ref{ODE3}) and taking (\ref{ODE+-}) into account, we get the explicit form
of the interface operator $\widehat F$
$$
\widehat F \psi= \delta \left[ (4 \alpha k^2+1) \psi_+'(0)-\psi_-'(0)  \right]  + \delta' \left[ (-4\alpha
k^2+1) \psi_+(0)-\psi_-(0)  \right]
$$
The interface condition can then be written
$$
\widehat F \psi=0 \Longleftrightarrow {\bf A} \overline{\psi_-(0)}={\bf B} \overline{\psi_+(0)}
$$
where $\overline{\psi_\pm(0)}=\left[\begin{array}{c}\psi_\pm(0) \\
\psi'_\pm (0) \end{array} \right]$ and
\begin{equation}\label{GM}
\bf A= \left[ \begin{array}{rr} 0 & 1 \\
1 & 0
\end{array} \right] \quad , \quad
\bf B= \left[ \begin{array}{cc} 0 & 4 \alpha k^2 + 1 \\
1-4 \alpha k^2 & 0
\end{array} \right]
\end{equation}

We now consider two different cases:

\subsubsection{First Case: Interacting conditions}\label{sub}

Let $k=1$, $\alpha= 1/8$. Then
\begin{equation}\label{GM1}
\bf A= \left[ \begin{array}{cc} 0 & 1 \\
1 & 0
\end{array} \right] \quad , \quad
\bf B= \left[ \begin{array}{cc} 0 & 3/2 \\
1/2 & 0
\end{array} \right] \, ,
\end{equation}
$\W= \mbox{Ran} \,{\bf A} \cap \mbox{Ran} \, {\bf B} =\CO^2 \text{ and } \mbox{Ker} \, {\bf A}= \mbox{Ker}
\,{\bf B}=\{0\}$. According to Theorems \ref{1iii} and \ref{1iv}, the condition ${\bf A}
\overline{\psi_-(0)}={\bf B} \overline{\psi_+(0)}$ is interacting and a solution of (\ref{ODE2}), in this case,
exists and is unique for arbitrary initial data given at $x_0 \not=0$. In fact
\begin{equation}\label{t}
{\bf A} \overline{\psi_-(0)}={\bf B} \overline{\psi_+(0)} \Longleftrightarrow \left\{ \begin{array}{ccc}
\psi_-'(0) &=& \frac{3}{2} \psi_+'(0)\\
&&\\
\psi_-(0) &=& \frac{1}{2} \psi_+(0)
\end{array} \right.
\end{equation}
and so it follows from (\ref{SODE2},\ref{ODE+-},\ref{h},\ref{t}) that the global solution of (\ref{ODE2}) is:
\begin{eqnarray}
\psi(x) &=& H_- (x) \left[\tfrac{3}{2}A \sin (x) + \tfrac{1}{2} B\cos (x) \right] + H(x) \left[A \sin (x)+B\cos (x)
\right] \nonumber \\
&& + \tfrac{3}{8}A \delta(x) - \tfrac{1}{8} B \delta'(x)
\end{eqnarray}
where the values of the constants $A,B \in \CO$ are completely fixed by the initial conditions
$\overline{\psi(x_0)}=\overline{C}$ given at $x_0 \not=0$. The order of the solution is $M-n=2$ when $B \not=0$,
and is $1$ when the initial conditions yield  $A \not=0$ and $B=0$ (this is consistent with the results of Theorem \ref{1i}).

\subsubsection{Second Case: Partially interacting conditions}\label{y}

Let $k=1$, $\alpha= 1/4$. Then
\begin{equation}\label{GM1}
\bf A= \left[ \begin{array}{lr} 0 & 1 \\
1 & 0
\end{array} \right] \quad , \quad
\bf B= \left[ \begin{array}{lr} 0 & 2 \\
0 & 0
\end{array} \right]
\end{equation}
and $\W=\mbox{Ran} \,{\bf A} \cap \mbox{Ran} \, {\bf B}=\mbox{Ran} \, {\bf B}= \{(x,y) \in \CO^2: \, y=0 \}$.
The interface condition
\begin{equation}\label{88}
{\bf A} \overline{\psi_-(0)}={\bf B} \overline{\psi_+(0)} \Longleftrightarrow \left\{ \begin{array}{ccc}
\psi_-'(0) &=& 2 \psi_+'(0)\\
\psi_-(0) &=& 0
\end{array} \right.
\end{equation}
is partially interacting since the values of $\overline{\psi_+(0)}$ are only partially fixed by those of
$\overline{\psi_-(0)}$ ($\psi_+'(0)$ is completely fixed, but $\psi_+(0)$ is free).

Taking into account (\ref{SODE2},\ref{h}) and the interface condition (\ref{88}), we get the global solution of
(\ref{ODE2}) for this case:
\begin{eqnarray}\label{81}
\psi(x)&=& H_- (x) \left[2B \sin (x) \right] + H(x) \left[A \cos (x)+B\sin (x) \right] \\
\nonumber&&+  \tfrac{3}{4}B \delta(x) - \tfrac{1}{4} A \delta'(x)
\end{eqnarray}
where $A,B \in \CO$ are integration constants.

It is clear from (\ref{81}) that a solution of (\ref{ODE2}) exists and is unique for arbitrary initial
conditions given at $x_0> 0$. On the other hand, for initial conditions given at $x_0 <0$, the solution of (\ref{ODE2}) exists
iff the solution $\psi_-$ of the associated IVP (\ref{ODE1}), with the same initial conditions, satisfies $ \overline{\psi_-(0)} \in \K_{\bf A} \Longleftrightarrow \psi_-(0)=0$. Moreover, if a solution of (\ref{ODE2}) exists, it
is not unique: there is a one-parameter family of solutions (parametrized by $A$) which are compatible with the
given initial condition. These properties are in accordance with the statement of Theorem \ref{1iv} for the case
$\mbox{Ker} \,{\bf A}=\{0\}$,  $\mbox{Ker} \,{\bf B}=\{(x,y) \in \CO^2: y=0\}$, $\K_{\bf A}=\{(x,y) \in \CO^2:
x=0\}$ and $\K_{\bf B}= \CO^2$.

\section{Generalized solutions of the EBB equation}\label{gg}

In this section we consider the EBB equation with a distributed vertical load $f$ and axial force $P$ (cf.
\cite{Ata97,HO07}):
\begin{equation}\label{importt}
 \left[ {a(x)  w''(x)}\right]'' + P(x) w''(x) = f(x),\quad x \in \left[ {-L,L} \right]
\end{equation}
$$
w(-L)=0,\quad w(L)=0, \quad w'(-L) = 0,\quad  w'(L) = 0
$$
where $w(x)$ is the transversal displacement of the beam axis and
$a(x)$ denotes the flexural stiffness, given by $a=EI$, where $E$
is the modulus of elasticity and $I$ the moment of inertia.
Moreover, the boundary conditions correspond to the case of a beam
that is clamped at both ends ($2L$ is the length of the beam).

The substitution $\psi= w'' $ is commonly used to lower the order of the equation (\ref{importt}):
\begin{equation}\label{importantegge}
\left[ a(x)\psi(x) \right]'' + P(x) \psi(x) = f(x),\quad  x \in \left[ {-L,L} \right]
\end{equation}
If $a(x)$ is non-differentiable or distributional then the term $a \psi$, and hence the differential equation
(\ref{importantegge}), is not in general well-defined for non-smooth functions $\psi$. Several approaches to
this case, using intrinsic products or generalized functions, have been presented in the literature, e.g.
\cite{BC07,Cad08,HO07,HO09,YSM00,YSR01}. In \cite{HO07} the authors studied the equation (\ref{importt}) for the
case where $a$ or $P$ display a jump discontinuity at an interior point of $[-L,L]$. They used the model product
\cite{Obe92} in order to define the term $a \psi$ precisely, and concluded that while the resulting formulation
is consistent with discontinuous parameters $a$ or $P$, it is not in general well-defined when either $a$ or $P$
are singular distributions. An alternative, in this case, is the formulation in terms of generalized functions
\cite{GKOS01,HO09,HKO13,Obe92}. Another possibility that has been considered in engineering applications is the
use of intrinsic products, specifically adapted to the particular form of the coefficients $a$ and $P$
\cite{Bag95,Bag02,BC07,Cad08}.

Here we will use the formalism of section 3 to obtain an intrinsic formulation of the EBB equation for the
general case $a, P \in \A$ and smooth $f$. In our formulation (\ref{importt}) is first rewritten in the form (\ref{eq0}):
\begin{equation}\label{nnn}
\left[{a_0(x)* w''(x)+ w''(x)*a_1(x)}\right]'' + P_0(x)* w''(x)+ w''(x)* P_1(x)= f(x) \, .
\end{equation}
We remark that if the coefficients $a_0,a_1,P_0,P_1 \in
\C_p^{\infty}$ then $w'' \in
\C_p^{\infty}$ (cf. Corollary \ref{5.1} below), and (\ref{nnn})
reduces to (\ref{importt}) with $a=a_0+a_1$ and $P=P_0+P_1$.

We will also consider the auxiliary equation (obtained from (\ref{nnn}) by setting $\psi= w''$):
\begin{equation}\label{mmm}
\left[ a_0(x)*\psi(x)+\psi(x)*a_1(x) \right]'' + P_0(x)* \psi(x)+\psi(x)*P_1(x)= f(x) \, .
\end{equation}
which will be important for studying (\ref{nnn}).

In the next subsection we present some general results concerning the regularity of the solutions of
eq.(\ref{nnn}). Then, in subsection 5.2, the eq.(\ref{nnn}) is used to model several different physical beams.
We consider the cases of: 1) Constant flexural stiffness, 2) Discontinuous flexural stiffness (corresponding to
a beam with two different sections), 3) Constant flexural stiffness with a structural crack, and finally 4)
Discontinuous flexural stiffness and a structural crack at the point of contact of the two sections. Up to our
knowledge, this last case has never been considered in the literature.

\subsection{General results}

Assume that (\ref{nnn}) is sectionally regular, and that the coefficients satisfy $a_0(x_0^-)+a_1(x_0^+) \not=
0$ at the non-regular points $x_0 \in \cup _{i=0}^1\text{sing supp } a_i$. As before, $a_i(x^\pm)=\lim_{y \to
x^\pm} a_i(y)$. Under these conditions, the next results are corollaries of Theorem \ref{1i}.

\begin{corollary} \label{5.1}
Let $a_i\in \C_p^{\infty}$, $i=0,1$ and let $P_i \in \A$ be such
that ord $P_i \le 2$, $i=0,1$. Then (\ref{nnn}) displays
continuously differentiable solutions.
\end{corollary}
\begin{proof}
Consider the auxiliary equation (\ref{mmm}). The maximal order of the coefficients is $M\le 2$ ($M$ may be the
order of $P_i$ or of $a_i''$). Hence $M \le n=2$, where $n$ is the order of the differential equation
(\ref{mmm}). It then follows from Theorem \ref{1i} that ord $\psi=0 \Longrightarrow
\psi\in \C_p^{\infty}$. Since $w''=\psi$, we conclude that $w \in \C^1$.
\end{proof}

An example of this form will be considered in the next section (cf. Fig. 1). Another case is:
\begin{corollary}\label{5.2}
Let $a_i \in \A$, $i=0,1$ be such that ord $a_i \le 1,\ i=0,1$ and
let $P_i \in \A$ be such that ord $P_i \le 3$, $i=0,1$. Then
(\ref{nnn}) has continuous solutions.
\end{corollary}
\begin{proof}
Consider, once again, the auxiliary equation (\ref{mmm}). The maximal order of the coefficients satisfies
$$
M= \, {\rm max} \{ \, {\rm ord} \, a_i'', \, {\rm ord }\, P_i \, , \, i=0,1 \} \le 3
$$
Hence, the solutions of (\ref{mmm}) satisfy ord $\psi \le M-2 =1$. It follows from $w''=\psi$ that $w$ is at
least continuous.
\end{proof}

Two examples of this form will also be considered in the next section (cf. Figs. 2, and 3). Finally:
\begin{corollary}
Let $a_i,P_i \in \A$, $i=0,1$ and let
$$
M= \, {\rm max} \{ \, 2+ {\rm ord} \, a_i, \, {\rm ord }\, P_i \, , \, i=0,1 \}
$$
If $M \le 4$ then the solution of (\ref{nnn}) satisfies ord $w =0 \Longleftrightarrow w \in \C_p^{\infty}$. If
$M > 4$ then ord $w \le M-4 $ and $w$ may be singular.
\end{corollary}
\begin{proof}
The proof follows directly from Theorem \ref{1i} and the fact that (\ref{nnn}) is a fourth order
differential equation.
\end{proof}

\subsection{A non-uniform beam with structural cracks}
We now study several particular examples of physical beams. Consider a clamped-clamped (CC) beam of length $2L$
that is made of two segments with (possible) different flexural stiffness, and may exhibit a crack at the point
of contact of the two sections. To simplify the formulation, the contact point of the two sections is assumed to
be the middle point of the beam (which is placed at $x=0$). This system can be modelled by eq.(\ref{nnn}) with
coefficients:
\begin{eqnarray} \label{cracks}
a_0(x) & = & A H_-*\left(1- K_0 \delta\left(\tfrac{x}{2L}\right) \right)=A H_- - 2 A L K_0 \delta(x) \nonumber
\\
a_1(x) &=& \left(1- K_1 \delta\left(\tfrac{x}{2L}\right)\right)* BH=B H -2BL K_1 \delta(x)
\end{eqnarray}
where $A >0$ and $B>0$ are the constant flexural stiffness in the sections $[-L,0)$ and $(0,L]$, respectively.
Following \cite{BC07,Cad08}, the crack was modelled by a Dirac delta term. The parameters $K_0,K_1$ are related
to the depth of the crack at the left and right sides of the contact point. If the crack is located at a regular
point of the beam (i.e. $A=B$) then the particular values of $K_0$ and $K_1$ are irrelevant and the solution of
eq.(\ref{nnn}) is only dependent of the value of $K_0+K_1$ (cf. eqs.(\ref{ic},\ref{S}) below).

In order to simplify the discussion, let us consider, in addition, that in (\ref{nnn}) the vertical load is
constant $f(x)=C$ and that there is no axial force. The equation (\ref{nnn}) then reads:
\begin{equation} \label{nnn1}
\left[(A H_- - 2 A L K_0 \delta(x) )* w'' + w'' * (B H -2BL K_1 \delta(x)) \right]''=C \, , \quad x \in [-L,L]
\end{equation}
and the equation for $\psi=w''$, is:
\begin{equation} \label{mmm1}
[A H_- * \psi + B \psi * H -2ALK_0 \delta(x) * \psi -2BL K_1 \psi * \delta (x)]''=C \, .
\end{equation}
On $[-L,0)$ and $(0,L]$, (\ref{mmm1}) reduces to
\begin{equation} \label{nn11}
A \psi''(x)= C \quad \mbox{and} \quad B \psi''(x)= C \, ,
\end{equation}
respectively. Hence, (\ref{mmm1}) is sectionally regular (cf. Definition \ref{Re}) and it follows from Theorem
\ref{1i} that its general solution is of the form:
\begin{equation} \label{nn12}
\psi =H_- \psi_- + H \psi_+ + \Delta
\end{equation}
where supp $\Delta \subseteq \{0\}$ and $\psi_-$, $\psi_+$ satisfy the first and second equation of (\ref{nn11})
on $[-L,0]$ and $[0,L]$, respectively. Substituting (\ref{nn12}) into (\ref{mmm1}), we get:
\begin{eqnarray}\label{nn13}
 && H_- A \psi''_- + H B \psi''_+ + (A+B) \Delta''  + \left[ B \psi_+'(0) - A \psi_-'(0) \right] \delta (x) \\
\nonumber &+& \left[B \psi_+(0)- A \psi_-(0) \right] \delta '(x) -2L\left[A K_0\psi_+(0) +BK_1 \psi_-(0)
\right]\delta''(x)=C \, .
\end{eqnarray}
Using (\ref{nn11}), the two first terms cancel the right-hand side. It follows that:
\begin{equation} \label{nn12-1}
\Delta= \frac{2L}{A+B} (AK_0\psi_+(0)  +B K_1 \psi_-(0)) \delta(x)
\end{equation}
and that
\begin{equation} \label{p1}
A \psi_-(0) = B \psi_+(0) \quad \mbox{and} \quad A \psi'_-(0) = B \psi'_+(0) \, .
\end{equation}
Hence, the interface conditions at $x=0$ are interacting.

Going back to equations (\ref{nn11}), we easily find their general
solutions:
\begin{eqnarray} \label{nn13}
\psi_-(x) &=& \tfrac{C}{2A} x^2 + \alpha_- x + \beta_- \nonumber \\
\psi_+(x) &=& \tfrac{C}{2B} x^2 + \alpha_+ x + \beta_+
\end{eqnarray}
where $\alpha_\pm$ and $\beta_\pm$ are integration constants. Collecting the results
(\ref{nn12},\ref{nn12-1},\ref{nn13}), we determine the explicit expression of $\psi$, and can then solve
$w''=\psi$. We obtain:
\begin{equation}\label{w1}
w(x)= H_- w_- + H w_+
\end{equation}
where
\begin{eqnarray} \label{w2}
w_-(x) &=& \tfrac{C}{24A} x^4 + \tfrac{\alpha_-}{6} x^3 + \tfrac{\beta_-}{2} x^2 + \gamma_- x + \epsilon_- \nonumber \\
w_+(x) &=& \tfrac{C}{24B} x^4 + \tfrac{\alpha_+}{6} x^3 + \tfrac{\beta_+}{2} x^2 + \gamma_+ x + \epsilon_+
\end{eqnarray}
and $\gamma_\pm$ and $\epsilon_\pm$ are new integration constants. Moreover, $w_\pm$ satisfy the interface
conditions
\begin{equation} \label{c1}
w_+(0)-w_-(0)=0 \quad , \quad w_+'(0)-w_-'(0)=\tfrac{2L}{A+B}(AK_0 \beta_+ + BK_1 \beta_-)
\end{equation}
(these can be easily obtained by substituting (\ref{w1},\ref{w2}) into $w''=\psi$). From the interface
conditions for $\psi$ (cf. (\ref{p1})), we also have
\begin{equation} \label{c2}
A \beta_-= B \beta_+ \quad \mbox{and} \quad A \alpha_-= B \alpha_+ \, .
\end{equation}

Imposing on (\ref{w1},\ref{w2}) the conditions (\ref{c1},\ref{c2}), the CC boundary conditions
$$
w_-(-L)=w_-'(-L)=w_+(L)=w_+'(L)=0 \, ,
$$
and solving for the integration constants, we get:
\begin{equation}\label{ic} \left\{ \begin{array}{lll}
\alpha_-  & = & \frac{(B-A)(3LC+S)}{8A(A+B)} \\
\\
\beta_-  & = & \frac{L}{12A} S \\
\\
\gamma_- & = & \frac{CL^3 (17A-B)+L^2S(7A+B)}{48A(A+B)}\\
\\
\epsilon_-&=& \frac{3CL^4+L^3S}{12(A+B)}
\end{array} \right. \, \, ,
\left\{ \begin{array}{lll}
\alpha_+ & = & \frac{(B-A)(3LC+S)}{8B(A+B)} \\
\\
\beta_+ & = & \frac{L}{12B} S \\
\\
\gamma_+ &= & \frac{CL^3 (-A+17B)+L^2S(A+7B)}{48B(A+B)}\\
\\
\epsilon_+ &=& \frac{3CL^4+L^3S}{12(A+B)}
\end{array} \right.
\end{equation}
where
\begin{equation}\label{S}
S=\tfrac{CL (A^2-34AB+B^2)}{A^2+14AB +B^2+8(A^2K_0+B^2K_1)} \, .
\end{equation}
The eqs.(\ref{w1},\ref{w2},\ref{ic}) yield the unique solution of eq.(\ref{nnn1}) for clamped-clamped boundary
conditions. It is interesting to notice that the solution $w$ is dependent of the individual values $K_0$ and
$K_1$ of the intensity of the crack at the left and right sides of the contact point, and not only of $K_0+K_1$.
The exception is when the beam is uniform ($A=B$).

\vspace{0.5cm}

Below we display the graphics of the deflection $w$ and the slope $w'$ for several different CC beams. In all
cases the length of the beam is $2L=500\, {\rm cm}$ and the vertical load is $C=-0.015 \, {\rm kN/cm}$.

\begin{enumerate}

\item[[Fig.1]] A uniform beam with no cracks (thin line) with parameters:
$$
A=B=10^8\, {\rm kNcm^2},\quad  K_0=K_1=0
$$
versus a non-uniform beam also without cracks (thick line) with parameters:
$$
A=2B=10^8\, {\rm kN cm^2}, \quad  K_0=K_1=0 \, .
$$
The deflection function is continuous and differentiable (cf. Corollary \ref{5.1}).

\vspace{0.5cm}

\item[[Fig.2]] The same uniform beam (thin line) versus a uniform beam with a structural crack at $x=0$ (thick
line):
$$
 A=B=10^8 \, {\rm kNcm^2},\quad  K_0=K_1=\tfrac{1}{5} \, .
$$
The intensity of the crack is given by $K_0+K_1$. As expected (cf. Corollary \ref{5.2}) the deflection function
is continuous but not differentiable at $x=0$.

\vspace{0.5cm}

\item[[Fig.3]] The same uniform beam (thin line) versus a non-uniform beam with a structural crack at the point
of contact between the two sections (thick line). The intensity of the crack is given by $K_0$ (on the left side
of the contact point) and $K_1$ (on the right side). The parameters of this beam are:
$$
A=2B=10^8\, {\rm kNcm^2}, \quad K_0= K_1=\tfrac{1}{5} \, .
$$
The solution is dependent on the particular values of $K_0$ and $K_1$ and not only on $K_0+K_1$ (cf.
eq.(\ref{S})).

\end{enumerate}

\vspace{1cm}

\hspace{-1cm}\includegraphics [scale=1] {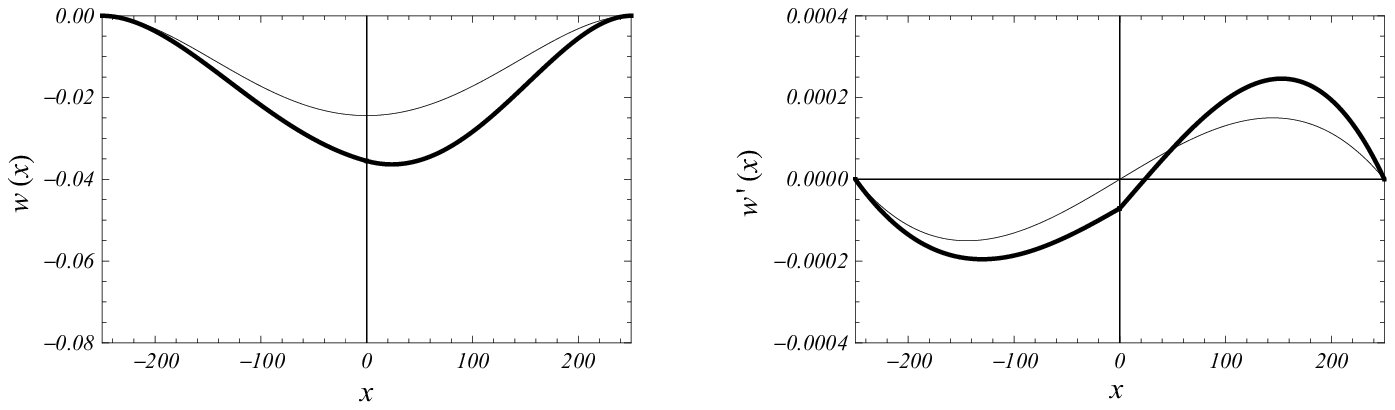}

\ \ \ \ \ \ \ \ \ \ \ \ \ \ \ \ \ \ \ \ \footnotesize {(a)}   \ \ \ \ \ \ \ \ \ \ \ \ \ \ \ \ \ \  \ \ \   \ \ \  \ \ \ \ \ \ \ \ \ \ \ \  \ \ \ \ \ \ \ \ \ \ \ \ \ \ \ \ \ \ \ \ \ \ \ \ \ \ \ \ \ \footnotesize  { (b)}\\
\footnotesize {Fig. 1: Deflection (Fig.1(a)) and slope (Fig.1(b))
of the CC uniform beam  ($\rule{10pt} {0,2pt}$) versus the CC
non-uniform beam ($\rule{10pt} {1,2pt}$),  both beams without
cracks.}

\vspace{2.5cm}
\hspace{-1cm}\includegraphics [scale=1] {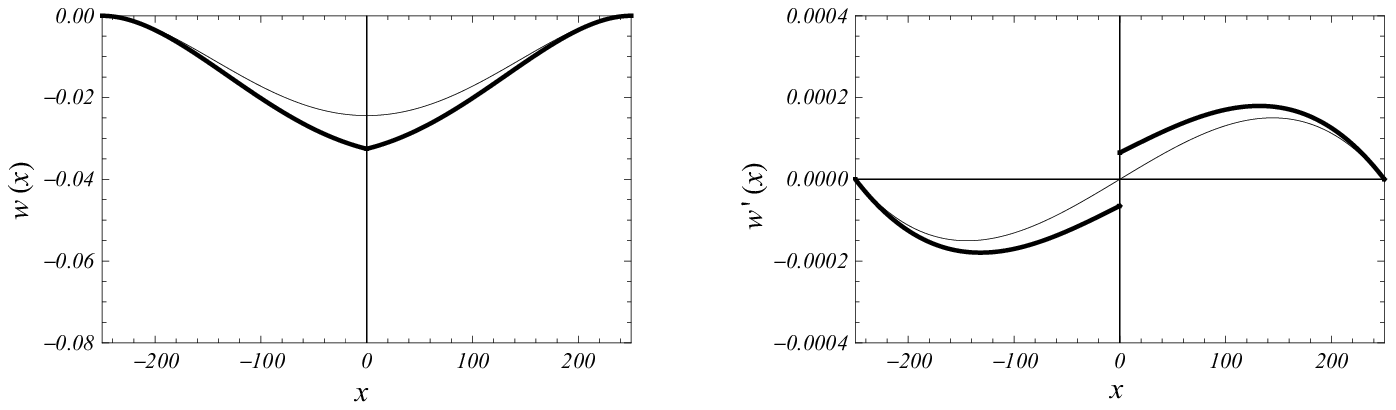}

\ \ \ \ \ \ \ \ \ \ \ \ \ \ \ \ \ \ \ \ \ \ \ \ \footnotesize {(a)}   \ \ \ \ \ \ \ \ \ \ \ \ \ \ \ \ \ \  \ \ \   \ \ \  \ \ \ \ \ \ \ \ \ \ \ \  \ \ \ \ \ \ \ \ \ \ \ \ \ \ \ \ \ \ \ \ \ \ \ \ \ \ \ \ \ \footnotesize  { (b)}\\
\footnotesize {Fig. 2: Deflection (Fig.2(a)) and slope (Fig.2(b))
of the CC uniform beam with no cracks ($\rule{10pt} {0,2pt}$)
versus the CC uniform beam with a structural crack at $x=0$
($\rule{10pt} {1,2pt}$).}

\vspace{2.5cm}
\hspace{-1cm}\includegraphics [scale=1] {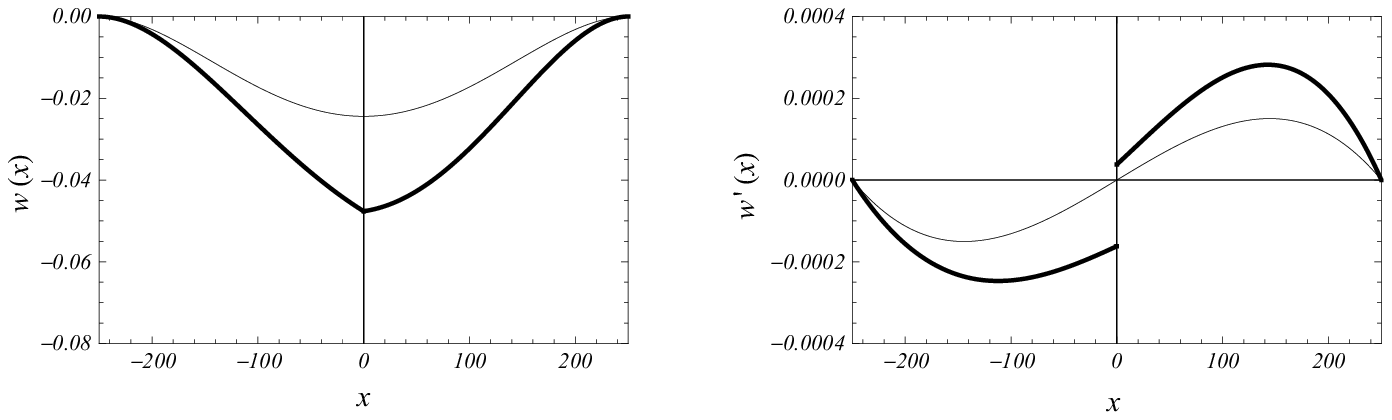}

\ \ \ \ \ \ \ \ \ \ \ \ \ \ \ \ \ \ \ \ \ \ \ \ \footnotesize {(a)}   \ \ \ \ \ \ \ \ \ \ \ \ \ \ \ \ \ \  \ \ \   \ \ \  \ \ \ \ \ \ \ \ \ \ \ \  \ \ \ \ \ \ \ \ \ \ \ \ \ \ \ \ \ \ \ \ \ \ \ \ \ \ \ \ \ \footnotesize  { (b)}\\
\footnotesize {Fig. 3: Deflection (Fig.3(a)) and slope (Fig.3(b))
of the CC uniform beam with no cracks ($\rule{10pt} {0,2pt}$)
versus the CC non-uniform beam with a structural crack at $x=0$
($\rule{10pt} {1,2pt}$).}

\normalsize

\bigskip

\noindent\textbf{Acknowledgements}. Cristina Jorge was supported
by the PhD grant SFRH/BD/85839/2012 of the Portuguese Science
Foundation. N.C. Dias and J.N. Prata were supported by the
Portuguese Science Foundation (FCT) under the grant
PTDC/MAT-CAL/4334/2014.

\vspace{1cm}

*******************************************************************

\setcounter{footnote}{0}

\textbf{Author's addresses:}

\begin{itemize}
\item \textbf{Nuno Costa Dias}\footnote{Corresponding Author} and \textbf{Jo\~ao Nuno Prata:}
Grupo de F\'{\i}sica Matem\'{a}tica, Departamento de Matem\'atica, Universidade
de Lisboa, Campo Grande, Edif\'{\i}cio C6, 1749-016 Lisboa, Portugal and
Escola Superior N\'autica Infante D. Henrique, Av. Eng. Bonneville Franco, 2770-058 Pa\c{c}o d'Arcos, Portugal.

\item \textbf{Cristina Jorge}: Departamento de Matem\'{a}tica.
Universidade Lus\'{o}fona de Humanidades e Tecnologias. Av. Campo
Grande, 376, 1749-024 Lisboa, Portugal and Grupo de F\'{\i}sica
Matem\'{a}tica, Departamento de Matem\'atica, Universidade
de Lisboa, Campo Grande, Edif\'{\i}cio C6, 1749-016 Lisboa, Portugal.

\end{itemize}

\vspace{0.3cm}

\small

{\it E-mail address} (NCD): ncdias@meo.pt

{\it E-mail address} (CJ): cristina.goncalves.jorge@gmail.com

{\it E-mail address} (JNP): joao.prata@mail.telepac.pt

*******************************************************************

\end{document}